\documentclass[11pt]{amsart}

\usepackage{amsmath}
\usepackage{amsfonts}
\usepackage{amssymb}
\usepackage{latexsym}

\usepackage{todonotes}
\usepackage{hyperref}

\newtheorem{theo}{Theorem}[section]
\newtheorem{lemma}[theo]{Lemma}
\newtheorem{coro}[theo]{Corollary}
\newtheorem{prop}[theo]{Proposition}

\def\N{\mathbb N}
\def\R{\mathbb R}

\newcommand{\Kn}{\mathcal{K}^n}
\newcommand{\ha}{\mathcal{H}}

\newcommand{\inte}{\mathrm{int}\,}

\newcommand{\op}{\overline\partial}

\newcommand{\aff}{\mathrm{aff}}
\newcommand{\lin}{\mathrm{lin}\,}
\newcommand{\U}{\mathrm{U}}
\newcommand{\V}{\mathrm{V}}
\newcommand{\SL}{\mathrm{SL}}
\numberwithin{equation}{section}

\begin{document}

\author{K\'aroly J. B\"or\"oczky} 
\address{Alfr\'eed R\'enyi Institute of Mathematics, Hungarian Academy
  of Sciences, Reltanoda u. 13-15., H-1053 Budapest, Hungary}
\email{carlos@renyi.hu}
\author{Martin Henk}
\address{OvG-Universit\"at Magdeburg, Fakult\"at f\"ur Mathematik,
  Universit\"atsplatz 2, D-39106 Magdeburg, Germany}
\email{martin.henk@ovgu.de}

\title{Cone-volume measure and stability}
\begin{abstract} 
We show  that the cone-volume measure of a convex body with centroid
at the origin satisfies the subspace concentration condition. This
implies, 
among others, a conjectured  best possible inequality for the
$U$-functional of a convex body. For both results we  provide 
stronger versions in the sense of stability inequalities.  
\end{abstract}
\date{\today}

\maketitle

\section{Introduction}
Let $\Kn$ be the set of all convex bodies in $\R^n$ having non-empty
interiors, 
i.e., $K\in \Kn$ is a convex compact subset of the
$n$-dimensional Euclidean space $\R^n$ with $\inte(K)\ne\emptyset$. As
usual, 
we denote by $\langle\cdot,\cdot\rangle$ the inner product on
$\R^n\times\R^n$ with  associated Euclidean norm
$\|\cdot\|$. $S^{n-1}\subset\R^n$ denotes the $(n-1)$-dimensional unit
sphere, i.e., $S^{n-1}=\{x\in\R^n : \|x\|=1\}$. The norm associated to
a $o$-symmetric convex body $K\in\Kn$ is denoted by $\|\cdot\|_K$,
i.e., $\|x\|_K=\min\{\lambda\geq 0: x\in\lambda\,K\}$.

For $K\in\Kn$, we write $S_K(\cdot)$ and $h_K(\cdot)$ to denote 
its surface area measure and support function, respectively,  
and $\nu_K$ to denote the Gau{\ss} map assigning  the exterior unit
normal $\nu_K(x)$ 
to an $x\in \partial_*K$, where $\partial_*K$  consists of all points in the
boundary $\partial K$ of $K$ having an unique outer normal
vector. If the origin $o$ lies in the interior of $K\in\Kn$,  
the {\em cone-volume measure of $K$} on $S^{n-1}$ is given by 
\begin{equation} 
\label{Gauss-cone}
\V_K(\omega)=\int_{\omega}\frac{h_K(u)}n\,dS_K(u)=
\int_{\nu_K^{-1}(\omega)}\frac{\langle x,\nu_K(x)\rangle}n\,d{\mathcal H}_{n-1}(x),
\end{equation}
where  $\omega\subset S^{n-1}$ is a Borel set and, in general,
${\mathcal H}_{k}(x)$ denotes the $k$-dimensional
Hausdorff-measure. Instead of ${\mathcal H}_n(\cdot)$, we also 
write $\V(\cdot)$ for the $n$-dimensional volume.  

The name cone-volume measure stems from the fact that if $K$ is a polytope with facets $F_1,\ldots,F_m$ and corresponding exterior unit normals $u_1,\ldots,u_m$, then
\begin{equation}
\V_K(\omega)=\sum_{i=1}^m\V([o,F_i])\delta_{u_i} (\omega).
\label{eq:cvm_polytope}
\end{equation}
Here $\delta_u$ is the Dirac delta measure  on
$S^{n-1}$ at $u\in S^{n-1}$, and for $x_1,\dots,x_m\in\R^n$ and 
subsets $S_1,\dots,S_L\subseteq \R^n$ we denote the convex hull of the
set $\{x_1,\dots,x_m, S_1,\dots,S_l\}$ by  $[x_1,\dots,x_m,
S_1,\dots,S_l]$.  With this notation $[o,F_i]$ is the cone with apex $o$ and
basis $F_i$.

In recent years, cone-volume measures have 
appeared and were studied in various contexts, see, e.g.,  
F. Barthe, O. Guedon, S. Mendelson and  A. Naor \cite{BGMN05}, 
K.J. B\"or\"oczky, E. Lutwak, D. Yang and G. Zhang \cite{BLYZ12,
  BLYZ13}, 
M. Gromov and V.D. Milman \cite{GrM87}, 
M. Ludwig \cite{Lud10}, 
M. Ludwig and  M. Reitzner \cite{LuR10}, 
E. Lutwak, D. Yang and G. Zhang \cite{LYZ05}, 
A. Naor \cite{Nao07}, 
A. Naor and D. Romik \cite{NaR03}, 
G.  Paouris and E. Werner \cite{PaW12}, 
A. Stancu \cite{Sta12}, 
G.~Zhu \cite{Zhu14a, Zhu14b}.

In particular, cone-volume measure are the subject of 
the  {\em logarithmic Minkowski problem}, which is the 
particular interesting limiting case $p=0$ of  the  general $L_p$-Minkowski problem -- one
of the central problems in  convex geometric analysis. It is the task: 

\smallskip \noindent 
{\em Find necessary and sufficient conditions for a Borel measure
  $\mu$ on $S^{n-1}$ to be the cone-volume measure $\V_K$ of 
  $K\in\Kn$ (with $o$ in its interior).}
\smallskip

In the  recent  paper \cite{BLYZ13}, 
K.J. B\"or\"oczky, E. Lutwak, D. Yang and G. Zhang characterize the cone-volume
measures of origin-symmetric convex bodies. In order to state their
result we have to introduce 
the  subspace concentration
condition.  We say that a Borel measure $\mu$ on $S^{n-1}$  satisfies
the {\em subspace concentration condition} if 
for any  linear subspace $L\subset \R^n$, we have
\begin{equation}
\label{scc}
\mu(L\cap S^{n-1})\leq \frac{\dim L}{n}\,\mu(S^{n-1}),
\end{equation}
and equality in (\ref{scc}) for some $L$ implies the existence of a complementary linear subspace $\widetilde{L}$ such that
\begin{equation}
\mu(\widetilde{L}\cap S^{n-1})=\frac{\dim\widetilde{L}}{n}\,\mu(S^{n-1}),
\label{eq:sccequality}
\end{equation}
and hence ${\rm supp}\,\mu\subset L\cup \widetilde{L}$, i.e., the
support of the measure ``lives'' in  $L\cup \widetilde{L}$.

Via the subspace concentration condition, the logarithmic Minkowski
problem was settled in \cite{BLYZ13}  in the
symmetric case.
  
\begin{theo}[\cite{BLYZ13}] A non-zero finite even Borel measure on
  $S^{n-1}$ is the cone-volume measure of an origin-symmetric convex
  bodies if and only if it satisfies the  subspace concentration
  condition.
\label{thm:cvm_symmetric}
\end{theo}

This result  was proved earlier for discrete measures  on $S^1$, i.e.,
for polygons, by A. Stancu \cite{Sta02, Sta03}.  
For cone-volume measures of origin-symmetric polytopes (cf.~\eqref{eq:cvm_polytope}) the necessity
of \eqref{scc} was independently shown by   M. Henk, A. Sch\"urmann
and J.M.Wills \cite{HSW05} and 
B. He, G. Leng and K. Li \cite{HLL06}.




We recall that the centroid of a $k$-dimensional convex compact set
$M\subset \R^n$ is defined as 
$$
c(M)={\mathcal H}_k(M)^{-1}\int_Mx\,d{\mathcal H}_k(x).
$$
The centroid seems also be the right and natural position of the origin
in order to extend Theorem \ref{thm:cvm_symmetric}  to arbitrary
convex bodies. In fact,   in \cite{HeL14} it was shown by  M. Henk and E. Linke
that the necessity part of Theorem  \ref{thm:cvm_symmetric} also holds for
polytopes with centroid at the origin, i.e.,  

\begin{theo}[\cite{HeL14}]
\label{Henk-Linke-polytope}
Let  $K\in \Kn$ be a  polytope with centroid at the origin.  Then its cone-volume
measure $\V_K$ satisfies the  subspace concentration condition.
\end{theo}  

Our first result is an extension of Theorem \ref{Henk-Linke-polytope}
to convex bodies. 
\begin{theo}
\label{Henk-Linke-body}
Let  $K\in\Kn$ with centroid at the origin. Then its cone-volume
measure satisfies the  subspace concentration condition.
\end{theo}  
 
While the  subspace concentration condition is also the sufficiency
property  to characterize  cone-volume measures among even non-trivial Borel
measures,  the cone-volume measure of a convex body $K\in \Kn$ whose
centroid is the origin should satisfy some extra properties. For
example, in Proposition~\ref{hemi-sphere} 
we prove that the measure of any open hemisphere is at least
$\frac1{2n}$. 


If the origin is the not the cetroid of the convex body, then the subspace
concentration condition may not hold anymore. 
In fact, it was recently shown by  G.~Zhu
\cite{Zhu14a} that for unit vectors $u_1,\dots,u_m\in S^{n-1}$ in general position, $m\geq
n+1$, and  arbitrary positive numbers 
$\gamma_1,\dots,\gamma_m$ there exists a polytope $P$ with outer unit
normals $u_i$ with $\V_P(\{u_i\})=\gamma_i$, $1\leq i\leq m$. In other
words, Zhu settled the logarithmic Minkowski problem for discrete
measures whose support is in general position. In general,
the centroid of such a polytope $P$ is not the origin, and a full characterization of
cone-volume measures of arbitrary polytopes/bodies is still a
challenging and important problem. 

We note that \eqref{eq:sccequality} is a kind of  condition   
on the cone-volume measure which is 
independent of the choice of the origin. 

\begin{lemma}
\label{general-split}
If $K\in\Kn$ with $o\in\inte K$, and 
${\rm supp}\,\V_K\subset L\cup\widetilde{L}$ for the proper complementary linear subspaces
 $L,\widetilde{L}\subset\R^n$, then  
$$
\V_K(L\cap S^{n-1})= \frac{\dim L}{n}\,\mu(S^{n-1}).
$$
\end{lemma}

Let us provide the simple argument leading to Lemma~\ref{general-split}. It follows from Minkowski's  uniqueness theorem that $K=M+\widetilde{M}$ where 
$M$, $\widetilde{M}$ are
contained in affine spaces orthogonal to $L$, $\widetilde{L}$, respectively. By
Fubini's theorem, we conclude  \eqref{eq:sccequality} for $\V_K$ and
the subspaces $L,\widetilde{L}$.

For a convex body $K$ containing the origin in its interior, 
E. Lutwak, D. Yang and  G. Zhang \cite{LYZ01} defined the $\SL(n)$
invariant quantity $\U(K)$ as  an integral over subsets
$(u_1,\ldots,u_n)\in S^{n-1}\times \cdots\times S^{n-1}$, by
$$
\U(K)=\left(\int_{u_1\wedge\ldots\wedge u_n\neq 0}d\V_K(u_1)\cdots d\V_K(u_n)\right)^\frac{1}{n},
$$
where $u_1\wedge\ldots\wedge u_n\neq 0$ means that the vectors
$u_1,\ldots,u_n$ are linearly independent. The $U$-functional has been proved useful in
obtaining strong 
inequalities for the volume of projection bodies \cite{LYZ01}. For 
information on projection bodies we refer to the books by Gardner
\cite{Gar95} and Schneider \cite{Sch93}, and for more
information on the importance of centro-affine functionals we refer to
C. Haberl and L. Parapatits  \cite{HaberlParapatits:2013, LuR10} and the references within.

We readily have $\U(K)\leq V(K)$, and equality holds if and only if
$V_K(L\cap S^{n-1})=0$ 
for any non-trivial subspace of $\R^n$
according to K.J. B\"or\"oczky, E. Lutwak, D. Yang and  G. Zhang
\cite{BLYZ14}. As a consequence of Theorem \ref{Henk-Linke-body} we prove here a
lower bound on $\U(K)$ in terms of $V(K)$ which was conjectured in 
 \cite{BLYZ14}.

\begin{theo}
\label{Henk-Linke-UK}
Let $K\in\Kn$ with centroid at the origin. Then 
$$
U(K)\geq \frac{(n!)^{1/n}}{n}\,V(K),
$$
with equality if and only if $K$ is a parallepiped.
\end{theo}  

In particular, $\U(K)>(1/\mathrm{e})V(K)$. For  polytopes,
Theorem~\ref{Henk-Linke-UK} was shown in \cite{HeL14}, where the
special cases if $K$ is an origin-symmetric polytope, or if $n=2,3$ were verified by
B. He, G. Leng and K. Li \cite{HLL06}, and G. Xiong \cite{Xio10}, respectively.

In order to state another consequence of Theorem~\ref{Henk-Linke-UK}
we need the notation of an {\em isotropic measure}, going  back to
K.M. Ball's reformulation of the Brascamp-Lieb inequality in
\cite{Bal91}. A Borel measure $\mu$ on $S^{n-1}$  is called  {\it
  isotropic} if 
\begin{equation*}
\label{isotropic}
{\rm Id}_{n}=\int_{S^{n-1}}u\otimes u\,d\mu(u), 
\end{equation*}
where ${\rm Id}_{n}$ is the $n\times n$-identity matrix and $u\otimes
u$ the standard tensor product, i.e., $u\otimes
u = u\,u^\intercal$.

Equating traces shows that for an isotropic measure $\mu(S^{n-1})=n$. 
The subspace concentration condition of a Borel measure $\mu$ on
$S^{n-1}$ is equivalent to have an {\it isotropic} normalized linear
image of $\mu$, i.e., that is,  there exists a  $\Phi\in{\rm GL}(n)$ such that
\begin{equation}
\label{lineariso}
{\rm Id}_{n}=\frac{n}{\mu(S^{n-1})}\int_{S^{n-1}}\frac{\Phi u}{\|\Phi u\|}\otimes \frac{\Phi u}{\|\Phi u\|}\,d\mu(u).
\end{equation}
The equivalence in this general  form is due to  K.J. B\"or\"oczky,
E. Lutwak, D. Yang and  G. Zhang \cite{BLYZ14}, while the
discrete case was earlier handled by E. A. Carlen, 
and D. Cordero-Erausquin \cite{CCE09}, and J. Bennett, A. Carbery,
M. Christ and T. Tao \cite{BCCT07} in their study of  the
Brascamp-Lieb inequality. Moreover,   the  case of a measure $\mu$ when strict inequality holds for all subspaces in \eqref{scc} is due to
 B. Klartag \cite{Kla10}.  Isotropic measures on $S^{n-1}$ are discussed also e.g. in 
F. Barthe \cite{Bar98,Bar04}, E. Lutwak, D. Yang and G. Zhang \cite{LYZ04, LYZ07}.
We note that isotropic measures on $\R^n$ play a central role in the
KLS conjecture by R. Kannan, L. Lov\'asz and   M. Simonovits
\cite{KLM95}, see, e.g.,  
F. Barthe and  D. Cordero-Erausquin \cite{BCE13}, O. Guedon and  E. Milman \cite{GuM11} and B. Klartag \cite{Kla09}.

Now from Theorem \ref{Henk-Linke-UK} and by the equivalence \eqref{lineariso} we
immediately conclude 
\begin{coro}  Every  convex body $K\in\Kn$   has an affine image,
  whose cone-volume measure is isotropic. 
\end{coro}
This, in particular, answers  a question posed by   
E. Lutwak, D. Yang and  G. Zhang \cite{LYZ05}.

In order to present stronger stability versions of Theorem
\ref{Henk-Linke-body} and Theorem \ref{Henk-Linke-UK} 
we need two notions of distance
between  the "shapes" of two convex bodies. Let  $K,M\in\Kn$,  and let
$K'=K-c(K)$, $M'=M-c(M)$
be their translates whose centroids are the origin.  Then we define 
\begin{eqnarray*}
\label{delta-hom-def}
\delta_{\rm hom}(K,M)&=&\min\{\lambda\geq 0:\exists t>0,
\;M'\subset t\, K'\subset \mathrm{e}^\lambda M'\},\\
\label{delta-vol-def}
\delta_{\rm vol}(K,M)&=&\frac{\V\big(M'\Delta t\,K'\big)}{\V(M)},\; t=\frac{\V(M)^{1/n}}{\V(K)^{1/n}},
\end{eqnarray*} 
where $A\Delta B$ denotes the symmetric difference of two sets, i.e.,
$A\Delta B = A\setminus B\, \cup\, B\setminus A$.

Then both $\delta_{\rm hom}$ and $\delta_{\rm vol}$  are metrics on the
space of convex bodies in $\R^n$ whose volumes are $1$, and centroids are the origin.

\begin{theo}
\label{Henk-Linke-stab} 
Let  $K\in\Kn$ with centroid at the origin, and let 
$$
\V_K(L\cap S^{n-1})> \frac{d-\varepsilon}{n}\,\V(K)
$$
for a non-trivial linear subspace
$L$ with $\dim L=d$ and  $\varepsilon\in(0,\varepsilon_0)$. Then there exist
 an $(n-d)$-dimensional compact convex set $C\subset L^\bot$, and
 a complementary $d$-dimensional compact convex set $M$
such that
$$
\delta_{\rm hom}(K,C+M)\leq \gamma_h\varepsilon^{1/(5n)}\mbox{ \ and \ }
\delta_{\rm vol}(K,C+M)\leq \gamma_v\varepsilon^{1/5},
$$
where $\varepsilon_0,\gamma_h,\gamma_v>0$ depend only on $n$.
\end{theo} 
Here $L^\bot$ denotes the orthogonal complement of $L$, and $M$ is
called a complementary compact convex set of $C$, if the linear spaces
generated by $M$ and $C$ are complementary.

Observe that the range of $\varepsilon$, i.e., $\varepsilon_0$, in Theorem \ref{Henk-Linke-stab} has
to depend  on the dimension. 
  For if,  let $K\in\Kn$ be a simplex whose centroid is the origin,
  and let $L$ be generated by $d$ outer normals of the simplex,
  $d\in\{1,\ldots,n-1\}$.  Then we have  $\V_K(L\cap S^{n-1})=\frac{d}{n+1}\,\V(K)$.

Actually, if $L$ is  $1$-dimensional, 
then a more precise version of Theorem~\ref{Henk-Linke-stab} holds. 

\begin{theo}
\label{Henk-Linke-stab1}
Let  $K\in\Kn$ with centroid at the origin, and let 
$$
\V_K(L\cap S^{n-1})> \frac{1-\varepsilon}{n}\,\V(K)
$$
for a  linear subspace
$L$ with $\dim L=1$ and $\varepsilon\in(0,\tilde{\varepsilon}_0)$. Then there exist
 $(n-1)$-dimensional compact convex set $C\subset L^\bot$ with
 $c(C)=o$, 
and $x,y\in \partial K$  such that $y=-\mathrm{e}^sx$ where $|s|<\tilde{\gamma_v}\varepsilon^{\frac1{6}}$,
$[x,y]+C\subset K$, and
$$
K\subset [x,y]+(1+\tilde{\gamma_h}\varepsilon^{\frac1{6n}})C
\mbox{ \ and \ }
\V(K) \leq (1+\tilde{\gamma_v}\varepsilon^{\frac1{6}})\V([x,y]+C),
$$
where $\tilde{\varepsilon}_0,\tilde{\gamma}_h,\tilde{\gamma}_v>0$
depend only on $n$.
\end{theo} 
We use this theorem in order to deduce the following stability version of Theorem~\ref{Henk-Linke-UK}.

\begin{theo}
\label{U(K)stab}
Let $K\in \Kn$ with centroid at  the origin, and let 
$$
\U(K)\leq (1+\varepsilon) \frac{(n!)^{1/n}}{n}\,\V(K)
$$
for $\varepsilon\in(0,\varepsilon_*)$. Then there exists  
a $K$ containing parallepiped $P$, such that for any facet $F$ of $P$, we have
$$
{\mathcal H}_{n-1}(F\cap K)\geq (1-\gamma_*\varepsilon^{\frac1{6}}){\mathcal H}_{n-1}(F),
$$
where $\varepsilon_*,\gamma_*>0$ depend only on $n$. In particular, we
have 
$$
(1-\gamma \varepsilon^{\frac1{6n}})P\subset K
\mbox{ \ and \ }
\V(P\backslash K)\leq \gamma \varepsilon^{\frac1{6}}\V(K).
$$
\end{theo}

The paper is organized as follows. In the next section we collect some
basic facts and notations from convexity which will be used later on. 
The third section is devoted to the proof of Theorem \ref{Henk-Linke-body}.
In Section \ref{another_property} we show another characetristic
property of  cone-volume measures of
convex bodies with centroid at the origin.  
The proofs of Theorem \ref{Henk-Linke-stab}, \ref{Henk-Linke-stab1}
are given in Section 8 and are prepared in  Sections 5--7.
Finally, in Section
\ref{secU(K)} we prove Theorem \ref{Henk-Linke-UK}.

\smallskip
{\it Acknowledgements.} We are grateful to Rolf Schneider for various
ideas shaping this paper. We also acknowledge fruitful discussions with
Daniel Hug and David Preiss about the Gau\ss-Green theorem.

\section{Preliminaries\label{pre}}
Good general references for the theory of convex bodies are provided
by the books of Gardner\cite{Gar95}, Gruber\cite{Gruberbook}, Schneider\cite{Sch93} and
Thompson\cite{Tho96}.

The support function $h_K : \R^n\to \R$   of convex body $K\in\Kn$ is
defined, for $x\in\R^n$, by
\begin{equation*}
  h_K(x)=\max\{\langle x,y\rangle : y\in K\}.
\end{equation*} 

A boundary point $x\in\partial K$  is said to have  a unit outer normal
(vector) $u\in S^{n-1}$ provided $\langle x,u\rangle= h_K
(u)$. $x\in\partial K$ is called singular if it has more than one unit
outer normal, and $\partial_* K$ is the set of all non-singular
boundary points. It is well known  that the set of singular boundary points of a convex body has $\mathcal{H}_{n-1}$-measure equal to $0$.
For each Borel set $\omega\subset S^{n-1}$, the inverse spherical
image of $\omega$  is the set of all
points of $\partial K$  which have an outer unit normal belonging to $\omega$. Since the
inverse spherical image of $\omega$  differs from
$\nu_K^{-1}(\omega)\subseteq\partial_* K$  by a set of $\mathcal{H}_{n-1}$-measure equal to $0$, we will often make no distinction between the two sets.

For $K\in \Kn$ the Borel measure $S_K$ on $S^{n-1}$ given by 
\begin{equation*}
 S_K(\omega)=\mathcal{H}_{n-1}(\nu_K^{-1}(\omega))
\end{equation*}
is called the (Aleksandrov-Fenchel-Jessen) surface area
measure. Observe that
\begin{equation*}
  \V(K)=\V_K(S^{n-1})= \int_{S^{n-1}}\frac{h_K(u)}{n}d\,S_K(u).
\end{equation*} 

As usual, for two subsets  $C, D\subseteq \R^n$ and reals
$\nu,\mu \geq 0$ the  Minkowski combination is defined by 
\begin{equation*}
  \nu\,C+\mu\,D = \{\nu\,c+\mu\,d : c\in C,\,d\in D\}.
\end{equation*}  
By the celebrated Brunn-Minkowski inequality we know that the $n$-th root 
of the volume of the Minkowski combination  is a concave function. 
 More precisely, for two convex compact sets 
$K_0,K_1\subset\R^n$ and for $\lambda\in[0,1]$  we have 
\begin{equation}
   \V((1- \lambda)\,K_0+\lambda\, K_1)^{1/n}\geq
   (1-\lambda)\,\V(K_0)^{1/n}+ \lambda\,\V(K_1)^{1/n}
\label{eq:brunn_minkowski} 
\end{equation}
with equality for some  $0<\lambda<1$ if and only if $K_0$ and $K_1$ lie
in parallel hyperplanes  or are homothetic, i.e., there exist 
$t\in\R^n$ and $\mu\geq 0$ such that $K_1=t+\mu\,K_0$  (see also
\cite{Gardnersurvey}).
 
Let $f:C\to \R_{> 0}$ be a positive function on an open  convex subset
$C\subset\R^n$ with the property that
there exists a $k\in\N$ such that $f^{1/k}$ is concave. Then by the
(weighted) arithmetic-geometric mean inequality 
\begin{equation*}
\begin{split}
f((1-\lambda)\, x+\lambda\, y) & = 
 \left(f^{1/k}((1-\lambda)\, x+\lambda\, y)\right)^k \\ &\geq  \left((1-\lambda)
 f^{1/k}(x)+\lambda f^{1/k}(y)\right)^k\\ &\geq f^{1-\lambda}(x)\cdot f^{\lambda}(y).
\end{split}
\end{equation*}
This means that $f$ belongs to the class of log-concave functions
which by the positivity of $f$ is equivalent to 
\begin{equation*}
  \ln f((1-\lambda)\, x+\lambda\, y)\geq (1-\lambda)\ln
  f(x)+\lambda\ln f(y) 
\end{equation*}   
for $\lambda\in [0,1]$. Hence,  for all $x,y\in C$ there exists a
subgradient $g(y)\in\R^n$ such that (cf., e.g., \cite[Sect. 23]{Rockafellar:1997ww})
\begin{equation}
\ln f(x)-\ln f(y)\leq\langle g(y), x-y\rangle.
\label{eq:gradient_concave}
\end{equation}
If $f$ is differentiable at $y$,  the subgradient  is the gradient
of $\ln f$ at $y$, i.e., $g(y)=\nabla \ln f = \frac{1}{f(y)}\nabla f(y)$.

 For a 
subspace $L\subseteq\R^n$, let  $L^\perp$ be its orthogonal
complement,  
and for $X\subseteq\R^n$ we denote by $X|L$ its orthogonal
projection onto $L$, i.e., the image of $X$ under the linear map
forgetting the part of $X$  belonging to $L^\perp$. 

Here, for a convex body $K\in\Kn$ and a $d$-dimensional subspace $L$,
$1\leq d\leq n-1$,  we are interested in the function measuring the volume of
$K$ intersected with planes parallel to $L^\perp$, i.e., in the function 
\begin{equation}
f_{K,L}: L \to \R_{\geq 0}\text{ with } x\mapsto \mathcal{H}_k(K\cap
(x+L^\perp)), 
\label{eq:function}
\end{equation} 
where $k=n-d$ is the dimension of $L^\perp$.  By
the Brunn-Minkowski inequality and the remark above, $f_{K,L}$ is a log-concave
on function on $K|L$ which is positive at least in   the relative interior of
$K|L$ 
(cf.~\cite{Ball:1988vo}). $f_{K,L}$ is also called the $k$-dimensional  X-ray of $K$
parallel to $L^\perp$ (cf.~\cite{Gar95}).  
By well-known properties of concave functions  we also know
\begin{prop}\hfill 
\begin{enumerate}
\item[{\rm i)}] $\ln f_{K,L}$ -- and thus $f_{K,L}$ -- is continuous on 
  $\inte(K)|L$. Moreover, $\ln f_{K,L}$ -- and thus $f_{K,L}$ -- are Lipschitzian on
  any compact subset of  $(\inte K)|L$. 
\item[{\rm ii)}] $\ln f_{K,L}$ -- and thus $f_{K,L}$ -- is on $\inte(K)|L$
  almost everywhere differentiable, i.e.,
  there exists a dense subset $D\subseteq \inte(K)|L$, where $\nabla
  f_{K,L}$ exists. 
\end{enumerate}
\label{rem:basicfacts_concave}
\end{prop} 
\begin{proof} For i) see, e.g., \cite[Theorem
  1.5.3]{Sch93}, and  for ii)    see, e.g., \cite[Theorem 
  25.5]{Rockafellar:1997ww}.
\end{proof}

Now for $K\in\Kn$ with centroid at $0$, i.e., $c(K)=0$, we have by Fubini's theorem with respect to the decomposition $L\oplus L^{\perp}$   
\begin{equation*}
\begin{split}
0 & =\int_K x\, \mathrm{d}\ha_n(x) \\ &= \int_{K|L}\left(
  \int_{(\hat{x}+L^\perp)\cap K} \tilde{x}\,
  \mathrm{d}\ha_k(\tilde{x})\right) \mathrm{d}\ha_d(\hat{x})\\ & = \int_{K|L}
   f_{K,L}(\hat{x})\, c((\hat{x}+L^\perp)\cap K)\,\mathrm{d}\ha_d(\hat{x}).
\end{split}
\end{equation*}
Writing $c((\hat{x}+L^\perp)\cap K)=\hat{x}+\tilde{y}$ with
$\tilde{y}\in L^\perp$ gives 
\begin{equation}
 \int_{K|L}  f_{K,L}(\hat{x})\,\hat{x}\,\mathrm{d}\ha_d(\hat{x}) =0.
\label{eq:centroid}
\end{equation}

\section{Proof of Theorem~\ref{Henk-Linke-body}\label{proof1}}
For the proof of Theorem \ref{Henk-Linke-body} we first establish some
more properties of the function $f_{K,L}$, where we always assume that
$L\subset\R^n$ is a $d$-dimensional linear subspace, $1\leq d\leq
n-1$,  with
$k$-dimensional orthogonal complement $L^\perp$.
We recall that
a function $f$ is said to be upper semicontinuous on $K|L$ if whenever
$x,y_m\in K|L$ for $m\in\N$  and $y_m$ tends to $x$, then
\begin{equation*}
          f(x)\geq \limsup_{m\to\infty} f(y_m).
\end{equation*}



\begin{lemma}
\label{upper-semicont}
The function $f_{K,L}$ is upper semicontinuous on $K|L$.
\end{lemma}
\begin{proof} Let $x,y_m\in K|L$ for $m\in\N$  be such that
  $\lim_{m\to\infty}y_m=x$. According to the Blaschke selection
  principle (cf., e.g., \cite{Sch93}), we may assume that the sequence of compact convex sets
$$
C_m=[(y_m+L^\bot)\cap K]-y_m\subset L^\bot
$$
tends to a compact convex set $C\subset L^\bot$ in the Hausdorff
topology. Since the $k$-volume of a compact convex set in $L^\bot$ is
a continuous functional, we have ${\mathcal
  H}_k(C)=\lim_{m\to\infty}f_{K,L}(y_m)$. However, $x+C\subset K$, and
therefore 
$f_{K,L}(x)\geq {\mathcal H}_k(C)$.
\end{proof}

An immediate consequence is that for sequences from the relative interior of
$K|L$, 
$f_{K,L}$ behaves ``continiuously'', i.e.,  
\begin{coro}
\label{ray-limit} Let  $o\in\inte K$ and $x\in K|L$.  Then $\lim_{m\to\infty}f_{K,L}(e^{\frac{-1}m}x)=f_{K,L}(x)$.
\end{coro}
\begin{proof} Since $o\in\inte K$, we get  by the concavity of
  $f_{K,L}^{1/k}$ that 
\begin{equation*}
f_{K,L}(e^{\frac{-1}m}x)\geq e^{\frac{-k}m}f_{K,L}(x).
\end{equation*}
 Since $f_{K,L}$ is also upper semicontinuous on $K|L$ by Lemma \ref{upper-semicont}, we conclude the corollary.
\end{proof}

Although the gradient $\nabla f_{K,L}$ might not be bounded, its norm belongs
  to the space $L^1(K|L)$ of absolute integrable functions.


\begin{lemma}
\label{nablafL1}
$\|\nabla f_{K,L}\|\in L^1(K|L)$, and thus the function $x\mapsto \langle \nabla f_{K,L}(x),x\rangle$ is in $L^1(K|L)$, as well.
\end{lemma}
\begin{proof} Let $f=f_{K,L}$.  Since $\nabla f(x) = \nabla
  (f^{\frac1k})^k(x) = kf^{\frac{k-1}k}(x)\nabla f^{\frac1k}(x)$ for
  almost all $x\in K|L$, it is  sufficient to prove 
$\|\nabla h\|\in L^1(K|L)$ for the concave function
$h=f^{\frac1k}$. However, by the Brunn-Minkowski theorem, the graph
$X$  of the function $h$ over $K|L$ is   part of the boundary of a
$(d+1)$-dimensional compact convex set. Thus
$$
\int_{K|L}\|\nabla h\|\,d{\mathcal H}_d(x) \leq \int_{K|L}\sqrt{1+\|\nabla h\|^2}\,d{\mathcal H}_d(x)={\mathcal H}_d(X)<\infty. \mbox{ \ }
$$
\end{proof}

The next two statements, which are the core ingredients of the proof of
Theorem~\ref{Henk-Linke-body} have been proved in the special case of  polytopes in
\cite{HeL14}.


\begin{prop}
\label{VKf}
If $o\in\inte K$, then
$$
n\,\V_K(L\cap S^{n-1})=d\,\V(K)+\int_{K|L}\langle \nabla f_{K,L}(x),x\rangle\,d{\mathcal H}_d(x).
$$
\end{prop}
\begin{proof}  Let $f=f_{K,L}$, and let $F(x)=f(x)x$ for $x\in K|L$, which is
a Lipschitz vector field on any compact subset of $(\inte K)|L$
(cf.~\eqref{rem:basicfacts_concave} i)). To state the 
Gau\ss-Green divergence theorem for Lipschitz vector fields on
Lipschitz domains, we follow  W.F. Pfeffer \cite{Pfe12}. Naturally,
$$
E_m=e^{\frac{-1}m}K|L\subset (\inte K)|L
$$ 
is a compact Lipschitz domain for $m\geq 1$, and hence $\op_\star
E_m=\op E_m$ according to Proposition~4.1.2 in \cite{Pfe12},
where 
$\op(E_m)$ denotes the (relative) boundary with respect to the linear space $L$.

Therefore Theorem~6.5.4  in \cite{Pfe12} (going back to H. Federer \cite{Fed45}) yields that
\begin{equation}
\label{divergence-theorem}
\int_{\op E_m}\langle F(x),\nu_{E_m}(x)\rangle\,d{\mathcal H}_{d-1}(x)=
\int_{E_m}{\rm div}F(x)\,d{\mathcal H}_d(x).
\end{equation}
If $y\in \op (K|L)$ then $\nu_{K|L}(y)=\nu_{E_m}(e^{\frac{-1}m}y)$; thus the left hand side of 
(\ref{divergence-theorem}) is
\begin{align*}
\int_{\op E_m}\langle F(x),\nu_{E_m}(x)\rangle\,d{\mathcal H}_{d-1}(x)=&
e^{\frac{-(d-1)}m}
\int_{\op (K|L)}\langle F(e^{\frac{-1}m}y),\nu_{K|L}(y)\rangle\,d{\mathcal H}_{d-1}(y)\\
=& e^{\frac{-d}m}
\int_{\op(K|L)}f(e^{\frac{-1}m}y)\langle y,\nu_{K|L}(y)\rangle\,d{\mathcal H}_{d-1}(y).
\end{align*}
Therefore, Corollary~\ref{ray-limit} and the Lebesgue dominated convergence theorem yield 
\begin{equation}
\lim_{m\to\infty}\int_{\op E_m}\langle F(x),\nu_{E_m}(x)\rangle\,d{\mathcal H}_{d-1}(x)
=\int_{\op (K|L)}f(y)\langle y,\nu_{K|L}(y)\rangle\,d{\mathcal H}_{d-1}(y).
\label{eq:stepone}
\end{equation}
Now, in order to evaluate the right hand side   
let $X=\partial K\cap (L^\bot+\op (K|L))$. Then  the set of
smooth points of $\partial K$ in $X$, i.e.,  $\partial_\star K\cap X$
coincides with the set of  points in $\nu_K^{-1}(L\cap S^{n-1})$. In addition,
if $z\in X\cap \partial_\star K$, then $\nu_{K|L}(y)=\nu_K(z)$ for
$y=z|L$, and  thus  \eqref{Gauss-cone} and \eqref{eq:stepone} give 
\begin{eqnarray}
\nonumber
\lim_{m\to\infty}\int_{\op E_m}\langle F(x),\nu_{E_m}(x)\rangle\,d{\mathcal H}_{d-1}(x)
&=&\int_{X}\langle z,\nu_{K}(z)\rangle\,d{\mathcal H}_{n-1}(z)\\
\label{left-div-theo}
&=&n\,\V_K(L\cap S^{n-1}).
\end{eqnarray}

Next, if $\nabla f(x)$ exists at  $x\in\inte(K)|L$, then 
$$
{\rm div}F(x)=d\,f(x)+\langle x,\nabla f(x)\rangle.
$$
Therefore the right hand side of \eqref{divergence-theorem} is
(cf.~Proposition \ref{rem:basicfacts_concave} ii), Lemma \ref{nablafL1}) 
$$
\int_{E_m}{\rm div}F(x)\,d{\mathcal H}_d(x)=d\int_{E_m}f(x)\,d{\mathcal H}_d(x)+
\int_{E_m}\langle x,\nabla f(x)\rangle\,d{\mathcal H}_d(x).
$$
Since $\int_{K|L}f(x)\,d{\mathcal H}_d(x)=\V(K)$, we deduce that
\begin{equation}
\label{right-div-theo}
\lim_{m\to\infty}\int_{E_m}{\rm div}F(x)\,d{\mathcal H}_d(x)=dV(K)+
\int_{K|L}\langle x,\nabla f(x)\rangle\,d{\mathcal H}_d(x).
\end{equation}
Combining \eqref{divergence-theorem}, \eqref{left-div-theo} and
\eqref{right-div-theo} completes the proof of the proposition.
\end{proof}


If $K$ is an $o$-symmetric convex body,  we know by the Brunn-Minkowski
inequality \eqref{eq:brunn_minkowski} that $f_{K,L}(x)$ attains its maximum at the origin
$0$. Hence, in view of \eqref{eq:gradient_concave} we know that
$\langle \nabla f_{K,L}(x),x\rangle\leq 0$ for almost every $x\in
K|L$. Although, this is no longer true for bodies with centroid at
$0$, the next proposition shows that it is true in the average.   

\begin{prop}
\label{fcentroid}
If $c(K)=o$, then
$$
\int_{K|L}\langle \nabla f_{K,L}(x),x\rangle\,d{\mathcal H}_d(x)\leq 0,
$$
with equality if and only if $f_{K,L}$ is constant on $K|L$.
\end{prop}
\begin{proof} Again, let $f=f_{K,L}$ and let $g:K|L\to L$ be a subgradient of $f$.  
For $z\in{(\rm int}K)|L$, applying \eqref{eq:gradient_concave} to
$y=o$ and $x=z$ first, and next to $y=z$ and $x=o$, we deduce that
\begin{equation}
\label{hprimef}
\langle g(z),z\rangle\leq \ln f(z)-\ln f(0)\leq \langle g(o),z\rangle, 
\end{equation}
where $g$ is a subgradient of $f$. In particular, if $\nabla f$ exists
at $z\in{(\rm int}K)|L$, then $\langle \nabla f(z),z\rangle\leq
\langle g(o),zf(z)\rangle$. Together with the property $c(K)=0$ we get
from \eqref{eq:centroid} that  
\begin{equation}
\label{fcentroid0}
\int_{K|L}\langle \nabla f(z),z\rangle\,d{\mathcal H}_d(z)\leq 
\int_{K|L}\langle g(o),zf(z)\rangle\,d{\mathcal H}_d(z)=0.
\end{equation}

Let us assume that equality holds in (\ref{fcentroid0}), and hence for
almost all $z\in (\inte K)|L$ in
(\ref{hprimef}). In particular, we have
$\ln f(x)-\ln f(0)=\langle g(o),x\rangle$, and in turn $f(x)=f(0)e^{\langle g(o),x\rangle}$
for almost all $x\in (\inte K)|L$. Since $f$ is continuous on $(\inte K)|L$, 
Corollary~\ref{ray-limit} yields that
$f(x)=f(0)e^{\langle g(o),x\rangle}$ for all $x\in K|L$. However $f^{\frac1k}$ is concave, therefore
$g(o)=o$, or in other words, $f$ is constant.
\end{proof}

Now, we are ready to give the proof of Theorem~\ref{Henk-Linke-body}. 
\begin{proof}[\bf Proof of Theorem~\ref{Henk-Linke-body}] Combining Propositions~\ref{VKf} and \ref{fcentroid} yields that
$$
V_K(L\cap S^{n-1})=\frac{d}n\,V(K)+\frac{1}n\,\int_{K|L}\langle \nabla f_{K,L}(x),x\rangle\,d{\mathcal H}_d(x)\leq
\frac{d}n\,V(K).
$$ 
Let us assume that  equality holds, and hence $f_{K,L}(x)=f_{K,L}(o)$ for $x\in K|L$ according to Proposition~\ref{fcentroid}.
Let $C(x)=K\cap (x+L^\bot)$ for  $x\in K|L$.
For any $x\in K|L$, there exists $\eta>0$ such that $-\eta x\in K|L$, and hence
$$
\mbox{$\frac{\eta}{1+\eta}\,C(x)+\frac{1}{1+\eta}\,C(-\eta x)\subset C(o).$}
$$
Therefore $f_{K,L}(x)=f_{K,L}(-\eta x)=f_{K,L}(o)$ and the equality
characterization of the 
Brunn-Minkowski inequality \eqref{eq:brunn_minkowski} implies  that $C(x)$ is a translate
of $C(o)$. 

Choose linearly independent $v_1,\ldots,v_d\in K|L$ such that
$v_0=-v_1-\ldots-v_d\in K|L$, as well. By $\sum_{i=0}^dv_i=o$ we have 
$\sum_{i=0}^d \frac1{d+1}\,C(v_i)\subset C(o)$, and we deduce that 
$\sum_{i=0}^d \frac1{d+1}\,c(C(v_i))=c(C(o))$. In particular, 
$$
c(C(o))\in\Pi=\aff\{c(C(v_0)),\ldots,c(C(v_d))\}, 
$$
where $\aff\{\}$ denotes the affine hull.
Next let $x\in K|L$. There exists $\eta>0$ such that $-\eta x\in[v_0,\ldots,v_d]$, and so
$\lambda x+\sum_{i=0}^d\lambda_i v_i=o$, where $\lambda+\sum_{i=0}^d\lambda_i=1$ and 
$\lambda,\lambda_i\geq 0$ for $i=0,\ldots,d$.
It follows as above that $\lambda\,C(x)+\sum_{i=0}^d\lambda_iC(v_i)=C(o)$,
and hence $c(C(x))\in\Pi$, as well. Therefore, writing
$M=\Pi \cap (K+L^\bot)$ and $C=C(o)-c(C(o))$, we get $K=C+M$. In particular,
${\rm supp}\,V_K\subset L\cup \Pi^\bot$ and $L\cap\Pi^\perp=\{0\}$.
\end{proof}

\section{Another property of the cone-volume measure
if the centroid is the origin \label{another_property}}

Let us recall two basic notions about convex bodies.
Firstly, a convex body in $\R^n$ is called a cylinder
if it is of the form $[p,q]+C$ for $p,q\in\R^n$ and an
$(n-1)$-dimensional convex compact set $C$;  $p+C$ and $q+C$ are
called bases of the cylinder.

 Secondly, let $v\in S^{n-1}$, and let $M$ be a convex body in $\R^n$.
For any $t$ with $-h_M(-v)<t<h_M(v)$,
we replace the section $M\cap(tv+v^\bot)$  with the $(n-1)$-ball
 of the same $(n-1)$-measure, centered at $tv$  in $tv+v^\bot$. Here,
 $v^\bot$ is the abbreviation for the linear space orthogonal to $v$. 

The closure $\widetilde{M}$ of the union
of these $(n-1)$-balls  is called the Schwarz rounding of $M$
with respect to $\R v$. It is a convex body 
by the Brunn-Minkowski theorem, and readily satisfies $\V(\widetilde{M})=\V(M)$.
If $\widetilde{M}$ is a cylinder, then all sections of the form $M\cap (tv+v^\bot)$
are of the same $(n-1)$-measure, and hence the equality
case of the Brunn-Minkowski theorem yields that $M$ is a cylinder, as
well. For more on Schwarz rounding we refer to \cite{Gruberbook}. 

\begin{prop}
\label{hemi-sphere} Let $K\in\Kn$ with $c(K)=o$ and $V(K)=1$. Then 
\begin{equation*}
         \V_K(\Omega)\geq \frac{1}{2n},
\end{equation*}
 for any open hemisphere $\Omega\subset S^{n-1}$. Equality holds if and
 only if $K$ 
is a cylinder whose generating segment is orthogonal
to the linear $(n-1)$-space bounding the hemisphere $S$. 
\end{prop}
\begin{proof} Let $\Omega\subset S^{n-1}$ be an open hemisphere, and
  let  $v\in S^{n-1}$ such that 
$$
\Omega=\left\{u\in S^{n-1}:\,\langle u,v\rangle>0\right\}.
$$
For any convex body $M\in\Kn$ with $o\in\inte M$  and
$x\in M|v^\bot$, let 
$$
f_M(x)=\max\{t\in\R\,:x+tv\in M\},
$$
and let $\varphi_M(x)=x+f_M(x)\,v$. 

In particular the points of $\partial M$ where
all exterior normals have acute angle with $v$ are of the form $\varphi_M(x)$
for $x\in \inte M|v^\bot$. Therefore
$$
\V_M(\Omega)=\V(\Xi_M)
\mbox{ \ for }\Xi_M=\bigcup_{x\in M|v^\bot}[o,\varphi_M(x)].
$$
For $x\in(\inte M|v^\bot)\backslash\{o\}$, let 
$z=\theta^{-1} x\in\partial M|v^\bot$ 
for some $\theta\in(0,1)$. Since $[\varphi_M(z),o,\varphi_M(o)]\subset \Xi_M$, we have

\begin{equation}
\begin{minipage}{0.9\hsize}
$x+\R v$ intersects $\Xi_M$
in a segment of length at least $(1-\theta)\|\varphi_M(o)\|$,
with equality if and only if $[\varphi_M(z), \varphi_M(o)]\subset\partial M$.
\end{minipage}
\label{eq:(i)}
\end{equation}

Now, let $\lambda=f_K(o)$, and hence $\lambda v\in\partial K$.  After a
linear transformation  
we may assume that
the tangent hyperplane $H$ at $\lambda v$ is given by  $H=\lambda v+v^\bot$.

We shake $K$ down to $H$, i.e.,  for each $x\in K|v^\bot$,
we translate the section $(x+\R v)\cap K$ by
$(\lambda-f_K(x))v$ 
and hence one endpoint lands in $H$. We write $K'$ to
denote the resulting convex body, which satisfies 
$$
K'|v^\bot=K|v^\bot=C-\lambda\,v 
\mbox{ \ for $C=K'\cap H$}.
$$

In addition $\V(K')=\V(K)=1$, and $\Xi_{K'}$ is the cone $[o,C]$.

For $x\in(\inte K|v^\bot)\backslash\{o\}$, it follows  
by \eqref{eq:(i)} that $x+\R v$ intersects $\Xi_K$
in a segment of length at least the length of
$\Xi_{K'}\cap(x+\R v)$. Therefore, Fubini's theorem yields
\begin{equation}
\label{K'K}
\V(\Xi_K)\geq \V(\Xi_{K'}).
\end{equation}
Furthermore,   Fubini's theorem implies that
\begin{eqnarray*}
\langle c(K'),u\rangle&=&\langle c(K),u\rangle=0
\mbox{ \ for $u\in v^\bot$};\\
\langle c(K'),v\rangle&\geq&\langle c(K),v\rangle=0
\mbox{ \ with equality if and only if $K'=K$}.
\end{eqnarray*}
We deduce
\begin{equation}
\label{K'c}
 c(K')=\eta v \mbox{ \ for $\eta\geq 0$, with $\eta=0$ if and only if $K'=K$}.
\end{equation}

Next let $\widetilde{K}$ be the
 Schwarz rounding of $K'$ with respect to $\R v$. 
It follows from the rotational symmetry
of  $\widetilde{K}$ that $\langle c(\widetilde{K}),u\rangle=0$
for $u\in v^\bot$, and by Fubini's theorem that
 $\langle c(\widetilde{K}),v\rangle=\langle c(K'),v\rangle$,
which in turn yield by (\ref{K'c}) and $\V(\widetilde{K})=\V(K')=1$ that
\begin{equation}
\label{tildeKc}
 c(\widetilde{K})=c(K')=\eta v
 \mbox{ \ for $\eta\geq 0$, with $\eta=0$ if and only if $K'=K$}.
\end{equation}
We conclude by (\ref{K'K}) and (\ref{tildeKc}) that
\begin{equation}
\label{tildeK}
\V(\Xi_K)\geq \V(\Xi_{\widetilde{K}-c(\widetilde{K})})
\mbox{ \ with equality if and only if $K'=K$}.
\end{equation}

Finally we compare $\widetilde{K}$ to the 
cylinder $Z$
over the $(n-1)$-ball $H\cap \widetilde{K}$, where $\V(Z)=\V(\widetilde{K})=1$
and $Z$ and $K$ lie on the same side of $H$. We deduce from the  
 rotational symmetry of $Z$  that $\langle c(Z),u\rangle=0$
for $u\in v^\bot$. On the other hand,  the rotational symmetry
of  $\widetilde{K}$ and $\widetilde{K}|v^\bot=(H\cap  \widetilde{K})-\lambda v$
yield that
$$
\langle x,v\rangle>-h_Z(-v)>\langle y,v\rangle
\mbox{ \ for all $x\in {\rm int}Z\backslash \widetilde{K}$
and $y\in \widetilde{K}\backslash Z$}.
$$
Therefore, 
$$
 c(Z)=\tau v
 \mbox{ \ for $\tau\geq\eta$, with $\tau=\eta$ if and only if $Z=\widetilde{K}$}.
$$
We conclude by (\ref{tildeKc}) and (\ref{tildeK}) that
$$
\V(\Xi_K)\geq \V(\Xi_{Z-c(Z)})=1/(2n)
\mbox{ \ with equality iff $K'=K$ and  $Z=\widetilde{K}$}.
$$
In turn, we get Proposition \ref{hemi-sphere}.
\end{proof}

\section{Some properties of the symmetric volume distance\label{prop_vol_distance}}

First we show that the distance $\delta_{\rm hom}$ can be estimated in terms of $\delta_{\rm vol}$.
These types of estimates have been around, only we were not able to locate them in the form we need.

\begin{lemma}
\label{centroidhom}
Let  $K\in\Kn$ with $c(K)=o$.
\begin{enumerate}
\item[{(i)}] If $Q\subset K$ is a convex body with $\V(K\backslash Q)\leq t\,\V(K)$ for $t\in(0,\frac1e)$, then
$(1-(et)^{1/n})K\subset Q$.
\item[{(ii)}] If $Q$ is a convex body with $\V(K\Delta Q)\leq t\V(K)$ for $t\in(0,\frac1{4^ne})$, then
$(1-(et)^{1/n})K\subset Q\subset (1+4(et)^{1/n})K$.
\end{enumerate}
\end{lemma}
\begin{proof} The main tool is the following result due to B. Gr\"unbaum \cite{Gru60}. If $M\in\Kn$,
and $H^+$ is a half space containing $c(M)$, then
\begin{equation}
\label{Grun}
\V(M\cap H^+)\geq \V(M)/e.
\end{equation}

To prove (i), let $\lambda=0$ if $o\not\in \inte Q$, and let $\lambda>0$ be maximal
with the property that $\lambda K\subset Q$ otherwise. In addition, let
$x=o$  if $o\not\in \inte Q$, and let $x$ be a common boundary point of $Q$ and $\lambda K$ otherwise.
Therefore, there exists a half space $H^+_1$ such that
$x$ lies on its boundary, and $H^+_1\cap \inte Q=\emptyset$. Now there exists a $y\in K$ such that
$x=\lambda y$, and hence $x$ is the centroid of $x+(1-\lambda)K=\lambda y+(1-\lambda)K\subset K$.
It follows from (\ref{Grun})  that
$$
t\V(K)\geq \V(H^+_1\cap K)\geq \V(H^+_1\cap (x+(1-\lambda)K))\geq
\V((1-\lambda)K)/e,
$$
and thus $t\geq \frac{(1-\lambda)^n}e$. 

To prove (ii), we observe that $\lambda K\subset Q$ for $\lambda=1-(et)^{1/n}$
by (i). We may assume that $Q\backslash  K\neq \emptyset$, and
let $\mu>1$ be minimal with the property that $Q\subset \mu K$. For a common boundary point $z$
of $Q$ and $\mu K$, let $w\in K$ such that $z= \mu w$. In particular, 
$w$ is the centroid of 
$$
w+\frac{\lambda(\mu-1)}{\mu}\,K\subset \frac{1}{\mu}\,z+\frac{\mu-1}{\mu}\,Q\subset Q.
$$
In addition there exists a half space $H^+_2$ such that
$w$ lies on its boundary, and $H^+_2\cap \inte K=\emptyset$. We
deduce again from (\ref{Grun}) that
$$
t\V(K)\geq \V(H^+_2\cap Q)\geq V\left(H^+_2\cap\left (w+\frac{\lambda(\mu-1)}{\mu}\,K\right)\right)\geq
\frac{\lambda^n(\mu-1)^n}{\mu^ne}\,\V(K).
$$
Now $t<\frac1{4^ne}$ yields that $\lambda>\frac12$ and $2(e\,t)^{1/n}<\frac12$, which in turn implies that
$\mu\leq (1-2(et)^{1/n})^{-1}<1+4(et)^{1/n}$. 
\end{proof}

\begin{coro}
\label{deltahomvol}
Let  $K, Q\in\Kn$. Then 
\begin{align*}
\delta_{\rm hom}(K,Q)&\leq 12\,\delta_{\rm vol}(K,Q)^{1/n}& \mbox{if $\delta_{\rm vol}(K,Q)<\frac1{4^ne}$,}\\
\delta_{\rm vol}(K,Q)&\leq  3n\,\delta_{\rm hom}(K,Q)&\mbox{if $\delta_{\rm hom}(K,Q)<\frac1{2n}$.}
\end{align*}
\end{coro}
\begin{proof} We use that $1+s<e^s<1+2s$  and $1-s<e^{-s}<1-\frac{s}2$ if $s\in(0,1)$.

We may assume that $c(K)=c(Q)=o$, and $\V(K)=\V(Q)=1$. In particular, $\V(K\Delta Q)=\delta_{\rm vol}(K,Q)$,
and hence the estimates for the exponential function and Lemma~\ref{centroidhom} yield with $s=\delta_{\rm vol}(K,Q)$ that
$$
e^{-2e^{1/n}s^{1/n}}K\subset(1-(se)^{\frac1n})K\subset Q\cap K \subset Q.
$$
Using the analogous  formula $e^{-2e^{1/n}s^{1/n}} Q\subset K$, 
we conclude the first estimate.

 For the second estimate, let $t= \delta_{\rm hom}(K,Q)$. It follows that
$e^{-t}K\subset Q\subset e^tK$, thus $\V(K\Delta Q)\leq
e^{nt}-e^{-nt}<3nt$. 
\end{proof}

Our next goal is  Lemma~\ref{nocentroid} stating
that one does not need to insist on the common centroid in the definition of  $\delta_{\rm vol}$. We prepare the argument
by the following observation.

\begin{lemma}
\label{translate-stab}
Let $K\in\Kn$ and $x\in\R^n$. Then 
\begin{equation*}
\V(K\Delta (x+K))\leq
2n\|x\|_{K-K}\V(K).
\end{equation*}
\end{lemma}
\begin{proof} We may assume that $x\neq o$.  Let $y,z\in K$ such that $x=\|x\|_{K-K}(y-z)$, and hence
$$
\|x\|_{K-K}=\|x\|/\|y-z\|.
$$
Applying Steiner symmetrization with respect to the hyperplane $x^\bot$ shows that
$$
\V(K)\geq \frac{\|y-z\|}n\,{\mathcal H}_{n-1}(K|x^\bot).
$$
We deduce by Fubini's theorem that
$$
\V(K\Delta (x+K))\leq 2\|x\|{\mathcal H}_{n-1}(K|x^\bot)\leq 2n\|x\|_{K-K}\V(K). 
$$
\end{proof}

\begin{lemma}
\label{nocentroid}
Let $K,Q\in\Kn$ with $c(K)=o$ and $\V(K\Delta Q)\leq t\V(K)$
for $t\in(0,\frac1{4^ne})$. Then 
\begin{equation*}
 \|c(Q)\|_{K-K}\leq 4n t  \quad\text{ and }\quad \delta_{\rm vol}(K,Q)\leq 9n^2 t. 
\end{equation*}
\end{lemma}
\begin{proof} We may assume that $\V(K)=1$, and the minimal volume so
  called L\"owner ellipsoid $E$ containing $K-K$ is a ball (see, e.g.,
  \cite{Gruberbook}).
In particular, $n^{-1/2}E\subset K-K\subset E$, and the Brunn-Minkowski and Rogers-Shephard theorems yield that
 $2^n\leq \V(K-K)\leq{2n \choose n}$. Since the volume of a centrally convex body over the volume of its Loewner ellipsoid 
is at least $2^n/(n!\V(B^n))$   according to K. Ball \cite{Bal91}, we have
 $$
2^n\leq \V(E)\leq {2n \choose n} \frac{n!}{2^n}\, \V(B^n)<\sqrt{3}\cdot\frac{2^nn^n}{e^n}\, \V(B^n).
$$
It follows that 
\begin{equation}
\label{K-Knorm}
\mbox{$\frac{2}{\sqrt{e\pi}}\,B^n\subset K-K\subset n B^n$ \ and \ 
$\frac1n\,\|x\|\leq\|x\|_{K-K}\leq 2\|x\|$.}
\end{equation}
Therefore, to prove Lemma~\ref{nocentroid}, it is sufficient to verify the corresponding estimate for $\|c(Q)\|$.

 If $c(Q)=o$, then we are done, otherwise let $u=c(Q)/\|c(Q)\|$. We have $Q\subset 2K\subset 2n B^n$ by Lemma~\ref{centroidhom} and  (\ref{K-Knorm}), and $\V(Q)\geq 1-t$ implies
$\V(Q)^{-1}<2$. By \eqref{K-Knorm} we also have 
\begin{equation*}
\|c(Q)\|_{K-K}\leq 2\|c(Q)\|=2\V(Q)^{-1}\langle  u,c(Q)\rangle=2\V(Q)^{-1}\left\|\int_Q\langle u,x\rangle\,dx\right\|,
\end{equation*}
and since $c(K)=o$ we get 
%
\begin{equation}
\begin{split}
\|c(Q)\|_{K-K}&\leq 2\V(Q)^{-1}\left\|\int_Q\langle
  u,x\rangle\,dx\right\|\\ 
&= 2\V(Q)^{-1}\left\|\int_{Q\backslash K}\langle
  u,x\rangle\,dx-\int_{K\backslash Q}\langle u,x\rangle\,dx\right\|\\ 
&4\int_{K\Delta Q} |\langle u,x\rangle|\,dx
\leq 4n t.
\end{split}
 \label{c(Q)norm}
\end{equation}


Let $K'=K+c(Q)$, thus Lemma~\ref{translate-stab} and (\ref{c(Q)norm}) imply that
$\V(K\Delta K')\leq 8n^2t$. We observe that $Q'=c(Q)+\V(Q)^{-1/n}(Q-c(Q))$ satisfies $c(Q')=c(Q)$, 
$\V(Q')=1$, and $\V(Q'\Delta Q)\leq t$ by $1-t\leq \V(Q)\leq 1+t$
(cf.~Lemma \ref{centroidhom}).
Therefore 
$$
\delta_{\rm vol}(K,Q)=\V(K'\Delta Q')\leq \V(K'\Delta K) +\V(K\Delta Q)+\V(Q\Delta Q')<9n^2t.
$$
\end{proof}

\section{Some consequences of the stability  of the Brunn-Minkowski inequality\label{sec_BM}}

Concerning the Brunn-Minkowski theory, including the properties of mixed volumes, the main reference is 
R.~Schneider \cite{Sch93}. We use the Brunn-Minkowski theory in $L^\bot$ in the terminology of Theorem~\ref{Henk-Linke-stab},
whose dimension is $k=n-d$.
For $k,m\geq 1$, let 
$$
I^k_m=\{(i_1,\ldots,i_m):\, i_j\in\N,\;j=1,\ldots,m\mbox{ and }i_1+\ldots+i_m=k\}.
$$
For compact convex sets $C_1,\ldots,C_m$ in $\R^k$ and $(i_1,\ldots,i_m)\in I^k_m$, the non-negative mixed volumes 
$\V(C_1,i_1;\ldots;C_m,i_m)$ were defined by H. Minkowski in a way such that if $\alpha_1,\ldots,\alpha_m\geq 0$, then
\begin{equation}
\label{mixed-volume}
{\mathcal H}_k\left(\sum_{j=1}^m\alpha_j C_j\right)=\sum_{(i_1,\ldots,i_m)\in I^k_m}
\V(C_1,i_1;\ldots;C_m,i_m)\alpha_1^{i_1}\cdot\ldots \cdot \alpha_m^{i_m}.
\end{equation}
The mixed volume $\V(C_1,i_1;\ldots;C_m,i_m)$ actually depends only on the $C_j$ with $i_j>0$, does not depend on the
order how the pairs $C_j,i_j$ are indexed, and we frequently ignore the pairs $C_j,i_j$ with $i_j=0$.
We have  $\V(C_1,k)={\mathcal H}_k(C_1)$, and $\V(C_1,i_1;\ldots;C_m,i_m)>0$ if each $C_j$ is $k$-dimensional. It follows by the Alexandrov-Fenchel inequality that
\begin{equation}
\label{Alexandrov-Fenchel}
\V(C_1,i_1;\ldots;C_m,i_m)^k\geq \prod_{j=1}^m{\mathcal H}_k(C_j)^{i_j}.
\end{equation}
An important special case of (\ref{Alexandrov-Fenchel}) is the classical Minkowski inequality, which says
\begin{equation}
\label{Minkowski-inequality}
\V(C_1,1;C_2,k-1)^k\geq {\mathcal H}_k(C_1){\mathcal H}_k(C_2)^{k-1}.
\end{equation} 
Equality holds for $k$-dimensional $C_1$ and $C_2$ in the Minkowski inequality (\ref{Minkowski-inequality}) if and only if
$C_1$ and $C_2$ are homothetic. We remark that the equality conditions in the Alexandrov-Fenchel inequality
(\ref{Alexandrov-Fenchel}) are not yet clarified in general.

Now the Alexandrov-Fenchel inequality (\ref{Alexandrov-Fenchel}), and actually already the Minkowski inequality 
(\ref{Minkowski-inequality}) yields the classsical (general) Brunn-Minkowski theorem  stating that
if $C_1,\ldots,C_m$ are compact convex sets in $\R^k$, and
$\alpha_1,\ldots,\alpha_m\geq 0$, then (cf.~\eqref{eq:brunn_minkowski})
\begin{equation}
\label{BrunnMink}
{\mathcal H}_k\left(\sum_{j=1}^m\alpha_j C_j\right)^{1/k}\geq
\sum_{j=1}^m\alpha_i{\mathcal H}_k(C_i)^{1/k}.
\end{equation}
Equality holds for $k$-dimensional $C_1,\ldots,C_m$ and positive $\alpha_1,\ldots,\alpha_m$ in the Brunn-Minkowski inequality (\ref{BrunnMink}) if and only if
$C_1$ and $C_j$ are homothetic for $j=2,\ldots,m$.

We need the following stability version of the Minkowski inequality  (\ref{Minkowski-inequality}) due to
A. Figalli, F. Maggi and A. Pratelli \cite{FMP10}.
 If  $C_1,C_2$ are $k$-dimensional compact convex sets in $\R^k$, and
\begin{equation}
\label{Minkowski-stab-cond}
\V(C_1,1;C_2,k-1)^k\leq (1+\varepsilon) {\mathcal H}_k(C_1){\mathcal H}_k(C_2)^{k-1}
\end{equation}
for small $\varepsilon\geq 0$, then \cite{FMP10} proves that
\begin{equation}
\label{Minkowski-vol-stab}
\delta_{\rm vol}(C_1,C_2)\leq\tilde{\gamma}_v\varepsilon^{1/2} 
\end{equation}
where the explicit $\tilde{\gamma}_v>0$ depends only on the dimension $k$.

We remark that here we only work out the estimate with respect 
to the symmetric volume distance $\delta_{\rm vol}$, and then just use
Corollary~\ref{deltahomvol} for $\delta_{\rm hom}$. Actually, V.I. Diskant \cite{Diskant} proved  that (\ref{Minkowski-stab-cond}) implies
\begin{equation}
\label{Minkowski-hom-stab}
\delta_{\rm hom}(C_1,C_2)\leq\tilde{\gamma}_h\varepsilon^{1/k}
\end{equation}
for an unknown $\tilde{\gamma}_h>0$ depending only on $k$.
We note that (\ref{Minkowski-vol-stab}) and Corollary~\ref{deltahomvol} readily yields a version of  (\ref{Minkowski-hom-stab})
with exponent $\frac1{2k}$ instead of $\frac1k$.\\

Combining the stability versions (\ref{Minkowski-vol-stab}) and (\ref{Minkowski-hom-stab}) with 
Lemma~\ref{translate-stab} and Lemma~\ref{nocentroid} leads to the following stability version of the Brunn-Minkowski inequality.

\begin{lemma}
\label{BrunnMinkowski-stab}
For any $k\geq 1$, $m\geq 2$ and $\omega\in(0,1]$, there exist positive $\varepsilon_0(k,m,\omega)$ and $\gamma(k,m,\omega)$ 
depending on $k$, $m$ and $\omega$ such that
if  $k$-dimensional compact convex sets $C_0,C_1,\ldots,C_m$  in $\R^k$, and $\alpha_1,\ldots,\alpha_m>0$
satisfy that  $\alpha_i/\alpha_j\geq\omega$  and 
${\mathcal H}_k(C_i)=V$ for $i,j=1,\ldots, m$, and
$$
\alpha_1C_1+\ldots+\alpha_mC_m\subset C_0
\mbox{ \ and \ }{\mathcal H}_k(C_0)\leq e^\varepsilon 
(\alpha_1+\ldots+\alpha_m)^kV
$$
for $\varepsilon\in(0,\varepsilon_0(k,m,\omega))$, then for  $i=1,\ldots, m$, we have
\begin{eqnarray*}
\delta_{\rm vol}(C_i,C_0)&\leq&\gamma(k,m,\omega)\varepsilon^{1/2},\\
\left\|c(C_0)-\sum_{i=1}^m\alpha_ic(C_i) \right\|_{C_0-C_0}&\leq &(\alpha_1+\ldots+\alpha_m)
\gamma(k,m,\omega)\varepsilon^{1/2}.
\end{eqnarray*}
\end{lemma}
\begin{proof} Since ${\mathcal H}_k(\alpha_1C_1+\ldots+\alpha_mC_m)\geq
(\alpha_1+\ldots+\alpha_m)^kV$ according to the Brunn-Minkowski inequality,
we may assume that
$\alpha_1C_1+\ldots+\alpha_mC_m=C_0$  by  Lemma~\ref{nocentroid}.
For  $1\leq i<j\leq m$, we apply the Alexandrov-Fenchel inequality (\ref{Alexandrov-Fenchel})
to each term in  (\ref{mixed-volume}) except for $ k\alpha_i\alpha_j^{k-1}\V(C_i,1;C_j,k-1)$ and deduce that
$$
k\alpha_i\alpha_j^{k-1}\V(C_i,1;C_j,k-1)\leq k\alpha_i\alpha_j^{k-1}V+(e^\varepsilon-1)(\alpha_1+\ldots+\alpha_m)^kV.
$$
Here $(\alpha_1+\ldots+\alpha_m)^k\leq (\frac{m}{\omega})^k\alpha_i\alpha_j^{k-1}$, and hence 
$$
\V(C_i,1;C_j,k-1)\leq\left(1+ \frac2k\left(\frac{m}{\omega}\right)^k \varepsilon\right) V.
$$
Thus (\ref{Minkowski-vol-stab}) yield
\begin{equation}
\label{vol-est}
\delta_{\rm vol}(C_i,C_j)\leq\bar{\gamma}(k,m,\omega)\varepsilon^{1/2}
\end{equation}
for $\bar{\gamma}(k,m,\omega)$ 
depending only on $k$, $m$ and $\omega$. To compare to $C_0$, we may assume that $V=1$,
$\alpha_1+\ldots+\alpha_m=1$ and
 $c(C_i)=o$ for $i=1,\ldots,m$. Let $M=C_1\cap\ldots \cap C_m$.

It follows from (\ref{vol-est}) that
$$
{\mathcal H}_k(C_i\backslash M)\leq  m\cdot \bar{\gamma}(k,m,\omega)\varepsilon^{1/2},\;i=1,\ldots,m,
$$
and hence ${\mathcal H}_k(M)\geq 1-m\cdot \bar{\gamma}(k,m,\omega)\varepsilon^{1/2}$.
Since $M\subset C_i$ for $i=1,\ldots,m$ yields $M\subset C_0=\sum_{i=1}^m\alpha_iC_i$, and ${\mathcal H}_k(C_0)\leq e^\varepsilon$, we deduce
$$
{\mathcal H}_k(C_0\Delta C_i)\leq 2\bar{\gamma}(k,m,\omega)\varepsilon^{1/2},\;i=1,\ldots,m.
$$
Therefore Lemma~\ref{translate-stab} and  Lemma~\ref{nocentroid}  imply
the required estimates for $\delta_{\rm vol}(C_i,C_0)$ and $c(C_0)$. 
\end{proof}

To prove the next Proposition~\ref{fBrunnMinkowski-stab}, we need the following observation.

\begin{lemma}
\label{simplex-inK}
If $M$ is a convex body in $\R^d$ such that $-M\subset \eta M$ for $\eta\geq 1$, then there exists an $d$-simplex
$T\subset M$ whose centroid is the origin such that $M\subset \eta d^{3/2}T$.
\end{lemma}
\begin{proof} We may assume that the John ellipsoid $E$ of maximal
  volume  contained in $M\cap(-M)$ is Euclidean ball, and let
  $T\subset M\cap (-M)$ be an inscribed regular simplex. Then
  $\eta^{-1}M\subset M\cap(-M)\subset\sqrt{d} E\subset d^{3/2}T$.
\end{proof}

For Proposition~\ref{fBrunnMinkowski-stab} we use the notation of the
previous sections, i.e.,  $K\in\Kn$ is a convex body  with $c(K)=o$, 
 $d,k\in\{1,\ldots,n-1\}$ with $d+k=n$,
and  $L$ is a $d$-dimensional linear subspace. 
For $x\in K|L$, we set 
$$
f(x)=f_{K,L}(x)={\mathcal H}_k(K\cap(x+L^\bot)).
$$

\begin{prop}
\label{fBrunnMinkowski-stab}
There exist $t_0,\gamma>0$ depending on $n$ with the following properties.
Let $t\in(0,t_0)$, let $M_*\subset K|L$ be a $d$-dimensional convex compact set, and let
$K_*=K\cap(M_*+L^\bot)$. If $e^{-t}\leq f(x)/f(o)\leq e^t$ holds for any  $x\in M_*$, then
 there exist a
 $k$-dimensional compact convex set $C\subset L^\bot$, and a complementary $d$-dimensional compact convex set $M$ such that
$$
\delta_{\rm vol}(K,C+M)\leq \gamma\max\left\{\frac{\V(K\backslash K_*)}{\V(K)},t^{1/2}\right\}.
$$
\end{prop}
\begin{proof} 
Since $c(K)=o$ we have $-K\subset nK$. Hence $-K|L\subset
nK|L$ and  we may choose, according to Lemma~\ref{simplex-inK}, 
$v_0,\ldots,v_d\in e^{-s}K|L$, for some $s>0$,
such that $v_0+\ldots+v_d=o$, and
\begin{equation}
\label{simplex-inKL}
e^{-s}K|L\subset n^{5/2}[v_0,\ldots,v_d].
\end{equation}
For $x\in e^{-s}K|L$, let $K(x)=K\cap(x+L^\bot)$, and let
\begin{equation}
\label{tildeKdef}
\widetilde{K}(x)=\frac{f(o)^{1/k}}{f(x)^{1/k}}\,K(x),\mbox{ \ and hence ${\mathcal H}_k(\widetilde{K}(x))=f(o)$.}
\end{equation}
We define 
\begin{eqnarray*}
A&=&{\rm aff}\{c(K(v_0)),\ldots,c(K(v_d))\},\\
M&=&\{y\in A: (y+L^\bot)\cap e^{-s}K\neq \emptyset\},\\
C&=&K(o)-c(K(o)).
\end{eqnarray*}
We compare $K_*$ to $M+C$. To this end we consider the affine bijection
$\varphi:L\to A$ defined by the correspondance 
$\{\varphi(x)\}=A\cap(x+L^\bot)$ for $x\in L$. In particular,
\begin{equation}
\label{phivio}
\varphi(v_i)=c(K(v_i)),\;i=0,\ldots,d\mbox{ \ and \ }\varphi(o)=\frac1{d+1}\sum_{i=0}^dc(K(v_i)).
\end{equation}

Let $x\in e^{-s}K|L$. We have $\frac{-1}{2n^{5/2}}x\in\frac12[v_0,\ldots,v_d]$ according to (\ref{simplex-inKL}), thus
$$
\frac{-1}{2n^{5/2}}\,x=\sum_{i=0}^d\alpha_iv_i\mbox{ \ where }
\sum_{i=0}^d\alpha_i=1\mbox{ and }\alpha_i\geq\frac1{2(d+1)},\;i=0,\ldots,d.
$$
We define
\begin{eqnarray*}
\tilde{\beta}&=&\frac{\beta f(x)^{1/k}}{f(o)^{1/k}} \mbox{ \ where \ }
\beta=\frac{1}{1+2n^{5/2}};\\
\tilde{\beta}_i&=&\frac{\beta_if(v_i)^{1/k}}{f(o)^{1/k}} 
\mbox{ \ where \ }\beta_i= \frac{\alpha_i2n^{5/2}}{1+2n^{5/2}},\;i=0,\ldots,d,
\end{eqnarray*}
and hence $\beta+\sum_{i=0}^d\beta_i=1$ and 
$\beta x+\sum_{i=0}^d\beta_iv_i=o$. The condition on the function $f$ yields that
$$
e^{-t/k}\leq \tilde{\beta}+\tilde{\beta}_0+\ldots+\tilde{\beta}_d\leq e^{t/k},
$$
and the ratio of any two of $\tilde{\beta},\tilde{\beta}_0,\ldots,\tilde{\beta}_d$ is at least $1/(4n^{5/2})$.
In particular, 
$$
e^t(\tilde{\beta}+\tilde{\beta}_0+\ldots+\tilde{\beta}_d)^k f(o)\geq{\mathcal H}_k(K(o)),
$$ 
and  the convexity of $K$ implies (cf.~\eqref{tildeKdef})
\begin{equation*}
\label{beta-x-vi}
\tilde{\beta}\widetilde{K}(x)+\sum_{i=0}^d\tilde{\beta}_i\widetilde{K}(v_i)=
\beta K(x)+
\sum_{i=0}^d\beta_iK(v_i)\subset K(o).
\end{equation*}
We deduce from  Lemma~\ref{BrunnMinkowski-stab}, the stability version
of the Brunn-Minkowski inequality,  that
there exists $\gamma^*>0$ depending on $n$ such that
for  $i=0,\ldots,d$, we have
\begin{eqnarray}
\label{o-vi-vol}
\delta_{\rm vol}(K(v_i),K(o)),\;\delta_{\rm vol}(K(x),K(o))&\leq&\gamma^*t^{1/2},\\
\label{o-vi-c}
\left\|c(K(o))-\beta c(K(x))-\sum_{i=1}^d\beta_ic(K(v_i)) \right\|_{K(o)-K(o)}&\leq &
\gamma^*t^{1/2}.
\end{eqnarray}
Naturally, if $k=1$, then even $t$ can be written instead of $t^{1/2}$ on the right hand side of (\ref{o-vi-vol}) and
(\ref{o-vi-c}), but we ignore this possibility.

First we asssume that $x=o$. In this case (\ref{phivio}) and (\ref{o-vi-c}) yield
\begin{equation}
\label{c(K(o))}
\left\|c(K(o))-\varphi(o) \right\|_{K(o)-K(o)}\leq 
\gamma^*t^{1/2}.
\end{equation}

Next let $x\in e^{-s}K|L$ be arbitrary. We have
$\beta \varphi(x)+\sum_{i=0}^d\beta_i\varphi(v_i)=\varphi(o)$ because $\varphi$ is affine.
We recall that $C=K(o)-c(K(o))$.
Let 
$$
w=c(K(o))-\beta c(K(x))-\sum_{i=1}^d\beta_ic(K(v_i)).
$$
Since $\beta \varphi(x)=\varphi(o)-\sum_{i=0}^d\beta_i\varphi(v_i)$, we have
\begin{equation*}
\begin{split}
 \big\|c(K(x)) & -\varphi(x) \big\|_{C-C} =  \frac{\left\|\beta c(K(x))-\beta \varphi(x)) \right\|_{C-C}}{\beta}\\
 &\leq \frac{\|\beta c(K(x))+w-\beta \varphi(x)) \|_{C-C}}{\beta}
 +\frac{\|-w\|_{C-C}}{\beta}\\
 &= \frac{\|c(K(o))-\varphi(o) -\sum_{i=0}^d\beta_i(c(K(v_i))-\varphi(v_i))\|_{C-C}}{\beta}\\
 &\,+\frac{\|c(K(o))-\beta
   c(K(x))-\sum_{i=1}^d\beta_ic(K(v_i))\|_{C-C}}{\beta}.
\end{split}
\end{equation*}
As $\varphi(v_i)=c(K(v_i))$ according to (\ref{phivio}), it follows by (\ref{o-vi-c}) and (\ref{c(K(o))}) that
\begin{equation}
\label{cK(x)}
\left\|c(K(x))-\varphi(x) \right\|_{C-C}\leq 
\frac{2\gamma^*}{\beta}\cdot t^{1/2}<6n^{5/2}\gamma^* t^{1/2}.
\end{equation}

For $x\in e^{-s}K|L$, we deduce in order from (\ref{cK(x)}), (\ref{o-vi-vol}) and (\ref{tildeKdef}) that
\begin{equation*}
\begin{split}
{\mathcal H}_k\big((C+&\varphi(x))\Delta K(x) \big)\leq
{\mathcal H}_k\big((C+\varphi(x))\Delta (C+c(K(x))) \big)\\
&+{\mathcal H}_k\big( (C+c(K(x)))\Delta(\widetilde{K}(x)-c(\widetilde{K}(x))+c(K(x))) \big)\\
&+{\mathcal H}_k\big( (\widetilde{K}(x)-c(\widetilde{K}(x))+c(K(x))) \Delta K(x)\big)\\
&< 9n^{5/2}\gamma^* t^{1/2} {\mathcal H}_k(C).
\end{split}
\end{equation*}
Hence, by Fubini's theorem we get 
$$
\V(K_*\Delta (M+C))<9n^{5/2}\gamma^* t^{1/2}\V(M+C)
$$
and Lemma~\ref{nocentroid} yields the required estimate for
$\delta_{\rm vol}$.
\end{proof}

\section{Some more properties of $f_{K,L}(x)$\label{prop_f}}
Here we establish some more properties of the log-concave function
(cf.~\eqref{eq:function}) 
\begin{equation*}
f_{K,L}: L \to \R_{\geq 0}\text{ with } x\mapsto \mathcal{H}_k(K\cap
(x+L^\perp)), 
\end{equation*} 
and use the notation as introduced in Section 2, i.e.,  $K\in\Kn$ is
an $n$-dimensional convex body with
$c(K)=0$, $L$ is a $d$-dimensional subspace $L$,
$1\leq d\leq n-1$,  and we set $k=n-d$. Since we will keep $K$ and $L$
fixed, we just write $f(x)$ instead of $f_{K,L}(x)$. As in Section 2
let $g(x)$ be the subgradient of $f(x)$, and we recall that
$g(x)=\nabla f(x)/f(x)$ almost everywhere on $\inte(K)|L$.

 For  $\eta\geq 0$, we set
\begin{eqnarray*}
M_\eta&=&\{x\in K|L: \ln f(x)-\ln f(o)\geq \langle g(o),x\rangle-\eta\},\\
K_\eta&= & K\cap(M_\eta+L^\bot).
\end{eqnarray*}
Since $\ln f$ is concave, both $M_\eta$ and $K_\eta$ are compact and convex. 

\begin{lemma}
\label{nablafstab} Let  $\eta\geq 0$. Then 
$$
\int_{K|L}\langle \nabla f(x),x\rangle d{\mathcal H}_d(x)\leq -\eta \V(K\backslash K_\eta).
$$
\end{lemma}
\begin{proof} Let $x\in(\inte K)|L$ and $\eta\geq 0$, and let us assume  
$\ln   f(x)-\ln f(o)\leq \langle g(o),x\rangle-\eta$. Then by  \eqref{eq:gradient_concave}
we have  $\langle g(x),x\rangle\leq \langle g(o),x\rangle-\eta$. 
Hence  if $\nabla f$ exists at $x\in(\inte K)|L$, then
\begin{eqnarray*}
\langle \nabla f(x),x\rangle&\leq &0 \mbox{ \ provided that $x\in M_\eta$,}\\
\langle \nabla f(x),x\rangle&\leq &\langle g(o),f(x)x\rangle-f(x)\eta
\mbox{ \ provided that $x\not\in M_\eta$}.
\end{eqnarray*}
We conclude the lemma by \eqref{eq:centroid}
and $\V(K\backslash K_\eta)=\int_{(K|L)\backslash M_\eta}f(x)\,dx$.
\end{proof}



\begin{lemma}
\label{fxf0eta}
Let $\eta\in[0,1]$. If $\V(K\backslash K_\eta)\leq \V(K)/(2^ne)$, then
$$
e^{-\tau}\leq \frac{f(x)}{f(o)}\leq e^{\tau}
\mbox{ \ for $\tau=7n^{3/2}\eta^{1/2}$ and $x\in M_\eta$.}
$$
\end{lemma}
\begin{proof} By Lemma~\ref{centroidhom} we have $\frac12\,K\subset
  K_\eta$, and $f(x)\geq f(o)e^{\langle g(o),x\rangle-\eta}$ for $x\in  K_\eta$. 
We claim that for $\pm y\in K_\eta$
\begin{equation}
\label{g(o)small}
|\langle g(o),y\rangle|\leq 3\sqrt{k\eta}.
\end{equation}
The concavity of $f^{1/k}$ yields that
\begin{eqnarray*}
f(o)^{1/k}&\geq& \frac{f(y)^{1/k}+f(-y)^{1/k}}2 
\geq f(o)^{1/k}e^{-\eta/k}\frac{e^{\langle g(o),y\rangle/k}+e^{\langle g(o),-y\rangle/k}}2 \\
&\geq & f(o)^{1/k}e^{-\eta/k} \left(1+\left(\frac{\langle g(o),y\rangle}{2k}\right)^{2} \right).
\end{eqnarray*}
Since $e^t<1+2t$ for $t\in[0,1]$, we conclude (\ref{g(o)small}).

It follows from $\frac12\,K\subset K_\eta$ and $-K\subset nK$ that
$\frac12\,(K|L)\subset M_\eta$ and $-(K|L)\subset n(K|L)$. 
In particular, if $x\in  M_\eta$ is arbitrary, then  
 $\pm y\in M_\eta$ for $y=\frac1{2n}\,x$. We deduce from (\ref{g(o)small})
that $|\langle g(o),x\rangle|=2n|\langle g(o),y\rangle|\leq 6n\sqrt{k\eta}$. 
Therefore, the lemma follows from 
 $f(o)e^{\langle g(o),x\rangle-\eta} \leq f(x)\leq f(o)e^{\langle g(o),x\rangle}$.
\end{proof}

\section{Proofs of Theorem~\ref{Henk-Linke-stab} and Theorem~\ref{Henk-Linke-stab1}\label{proof23}}

For the proofs of the two stability theorems \ref{Henk-Linke-stab} and
\ref{Henk-Linke-stab1}, let $K\in\Kn$ with $c(K)=o$, and let 
$$
\V_K(L\cap S^{n-1})> \frac{d-\varepsilon}{n}\,\V(K)
$$
for a non-trivial linear subspace
$L$ with $\dim L=d$ and $\varepsilon\in(0,(2^ne)^{-5})$. As before,
for $x\in K|L$ let
$$
f(x)={\mathcal H}_k(K\cap(x+L^\bot)).
$$
According to Proposition~\ref{VKf}, the condition on $\V_K(L\cap S^{n-1})$ is equivalent with
\begin{equation}
\label{HenkLinke}
\int_{K|L}\langle \nabla f(x),x\rangle d{\mathcal H}_d(x)>-\varepsilon \V(K).
\end{equation}

\begin{proof}[Proof of Theorem~\ref{Henk-Linke-stab}] We set $\eta=\varepsilon^{4/5}$, and use the notation of Lemma~\ref{nablafstab}. 
It follows from (\ref{HenkLinke}) and Lemma~\ref{nablafstab}  that
\begin{equation*}
\label{Ketacomplement}
 \V(K\backslash K_\eta)<\varepsilon^{1/5} \V(K)< \V(K)/(2^ne),
\end{equation*}
and from Lemma~\ref{fxf0eta} that
\begin{equation*}
\label{Metaf}
e^{-t}\leq \frac{f(x)}{f(o)}\leq e^{t}
\mbox{ \ for $t=7n^{3/2}\varepsilon^{2/5}$ and $x\in M_\eta$.}
\end{equation*}

We assume that $\varepsilon$ is small enough in order to
apply Proposition~\ref{fBrunnMinkowski-stab} with $M_*=M_\eta$ and $t=7n^{3/2}\varepsilon^{2/5}$.
We deduce the existence of an
 $(n-d)$-dimensional compact convex set $C\subset L^\bot$, and complementary $d$-dimensional compact convex set $M$ 
such that
$$
\delta_{\rm vol}(K,C+M)\leq \gamma_v\varepsilon^{1/5}.
$$
In turn Corollary~\ref{deltahomvol} implies that
$$
\delta_{\rm hom}(K,C+M)\leq \gamma_h\varepsilon^{1/(5n)},
$$
completing the proof of Theorem~\ref{Henk-Linke-stab}.
\end{proof}

\begin{proof}[Proof of Theorem~\ref{Henk-Linke-stab1}] Here we have  $d=1$. 
We may assume that $L=\R$, and
$K|L=[-a,b]$ where $0<a\leq b$. Since $c(K)=o$ implies $-K\subset nK$ according to B. Gr\"unbaum \cite{Gru60},
we have $b\leq na$.

We set $\eta=\varepsilon^{2/3}$, and use again the notation of Lemma~\ref{nablafstab}. We deduce  from (\ref{HenkLinke}) and Lemma~\ref{nablafstab}  that
\begin{equation}
\label{Ketacomplement1}
 \V(K\backslash K_\eta)<\varepsilon^{1/3} \V(K)< \V(K)/(2^ne),
\end{equation}
and from Lemma~\ref{fxf0eta} that
\begin{equation}
\label{Metaf1}
e^{-t}\leq \frac{f(x)}{f(o)}\leq e^{t}
\mbox{ \ for $t=7n^{3/2}\varepsilon^{1/3}$ and $x\in M_\eta$.}
\end{equation}
It follows from
Lemma~\ref{centroidhom} and  (\ref{Ketacomplement1}) that  
$\frac12[-a,b]\subset M_\eta$, therefore the concavity of $\ln f$ and (\ref{Metaf1}) yield that
\begin{equation}
\label{abfup}
f(x)\leq e^{2t}f(o) \mbox{ \ for $x\in[-a,b]$.}
\end{equation}
Let $M_\eta=[-a_\eta,b_\eta]$ for $a_\eta,b_\eta>0$.
Since $K\backslash K_\eta$ contains two cones, one with base $K(-a_\eta)$ and height $a-a_\eta$, and
one with base $K(b_\eta)$ and height $b-b_\eta$, we get by
(\ref{Metaf1}), (\ref{Ketacomplement1}) and (\ref{abfup}) that 
\begin{eqnarray*}
\frac{a-a_\eta+b-b_\eta}n\,e^{-t}f(o)&\leq &\frac{a-a_\eta+b-b_\eta}n(f(-a_\eta)+f(b_\eta))\\
&\leq& \V(K\backslash K_\eta)<\varepsilon^{\frac13} \V(K)\leq\varepsilon^{\frac13}e^{2t}f(o)(a+b).
\end{eqnarray*}
In particular, 
$$
{\mathcal H}_1(M_\eta)=a_\eta+b_\eta>(1-2n\varepsilon^{\frac13})(a+b).
$$
Here and below $\gamma_1,\gamma_2,\ldots$ denote positive constants depending on $n$.
We deduce by (\ref{Metaf1}) that if $\varepsilon$ is small enough, then
\begin{eqnarray*}
af(-a)+bf(b)&=&n\V_K(L\cap S^{n-1})>(1-\varepsilon)\V(K)>(1-\varepsilon){\mathcal H}_1(M_\eta)e^{-t}f(o)\\
&>& (1-\gamma_1\varepsilon^{\frac13})(a+b)f(o).
\end{eqnarray*}
Since $b\geq a$ and  $\frac{a}{a+b}\geq\frac1{n+1}$ by $b\leq na$, \eqref{abfup} implies that if $\varepsilon$ is small enough, then
$$
f(-a),f(b)\geq (1-\gamma_2\varepsilon^{\frac13})f(o).
$$
As $\ln f$ is concave, we have
\begin{equation*}
\label{fablow}
f(x)\geq (1-\gamma_2\varepsilon^{\frac13})f(o) \mbox{ \ for $x\in[-a,b]$.}
\end{equation*}
However $\frac{a}{a+b}\,C(b)+\frac{b}{a+b}\,C(-a)\subset C(o)$, where
$C(x)=K\cap(x+L^\perp)$.  Thus Lemma~\ref{BrunnMinkowski-stab} yields that
\begin{equation}
\label{KoKend}
\delta_{\rm vol}(C(o),C(-a))\leq \gamma_3\varepsilon^{\frac1{6}}\mbox{ \ and \ }
\delta_{\rm vol}(C(o),C(b))\leq \gamma_3\varepsilon^{\frac1{6}}.
\end{equation}
Hence, with 
$$
\widetilde{C}=(C(-a)-\tilde{x})\cap (C(b)-\tilde{y}) \mbox{ \ for $\tilde{x}=c(C(-a))$ and $\tilde{y}=c(C(b))$.}
$$
It follows from (\ref{abfup}) and (\ref{KoKend}) that
$$
[\tilde{x},\tilde{y}]+\widetilde{C}\subset K\mbox{ \ and \ }\V(K)\leq 
(1+\gamma_4\varepsilon^{\frac1{6}})\V([\tilde{x},\tilde{y}]+\widetilde{C}).
$$
Using Lemma~\ref{nocentroid}, we replace $\widetilde{C}$ by a suitably smaller homothetic copy $C$ such that $c(C)=o$, and obtain that
there exist $x\in \tilde{x}+\widetilde{C}$ and $y\in \tilde{y}+\widetilde{C}$ satisfying $o\in [x,y]$,
$e^{-s}\|x\|\leq \|y\|\leq e^{s}\|x\|$ for $s=\gamma_5\varepsilon^{\frac1{6}}$, and
$$
[x,y]+C\subset K\mbox{ \ and \ }\V(K)\leq 
(1+\gamma_6\varepsilon^{\frac1{6}})\V([x,y]+C).
$$
Finally, if $z\in [-a,b]$, then $-z/n\in[-a,b]$ and
$\frac{1}{n+1}\,C(z)+\frac{n}{n+1}\,C(-z/n)\subset C(o)$. Therefore Lemma~\ref{centroidhom},  
Lemma~\ref{BrunnMinkowski-stab} and the estimates above imply
$$
K\subset [x,y]+(1+\gamma_5\varepsilon^{\frac1{6n}})C,
$$
completing the proof of Theorem~\ref{Henk-Linke-stab1}.
\end{proof}

\section{Stability of the $\U$-functional  $\U(K)$ \label{secU(K)}}

Let $m\in\{1,\ldots,n\}$. In this section, a finite sequence
$u_1,\ldots,u_m$ always denote points of $S^{n-1}$, and by $\lin\{X\}$
we denote the linear hull of a set $X$.
As in  \cite{HeL14},  we define $\sigma_m(K)>0$ by
$$
\sigma_m(K)^m=\int_{u_1\wedge\ldots\wedge u_m\neq 0}1\,d\V_K(u_1)\cdots d\V_K(u_m).
$$
In particular, $\sigma_1(K)=\V(K)$, $\sigma_n(K)=U(K)$, and for $m<n$, we have 
\begin{equation}
\begin{split}
\label{sigmainduction}
&\sigma_{m+1}(K)^{m+1}=\\ 
&\,\int_{u_1\wedge\ldots\wedge u_m\neq 0}
\left(\V(K)-\V_K(S^{n-1}\cap{\rm lin}\{u_1,\ldots,u_m\})\right)
d\V_K(u_1)\cdots d\V_K(u_m).
\end{split}
\end{equation}
As $\V_K(S^{n-1}\cap{\rm lin}\{u_1,\ldots,u_m\})\leq \frac{m}n\,\V(K)$ 
for linearly independent $u_1,\ldots,u_m$ according to Theorem~\ref{Henk-Linke-body}, we deduce that
\begin{equation}
\label{sigmainduction0}
\sigma_{m+1}(K)^{m+1}\geq \left(1-\frac{m}n\right)\V(K)\sigma_m(K)^m.
\end{equation}
Therefore the inequality of Theorem~\ref{Henk-Linke-UK} follows from
$$
\U(K)^n\geq \frac1n\,\V(K)\sigma_{n-1}(K)^{n-1}\geq \ldots\geq \frac{(n-1)!}{n^{n-1}}\,\V(K)^{n-1}\sigma_1= \frac{n!}{n^n}\,\V(K)^n.
$$

Now we assume that 
$$
\U(K)\leq (1+\varepsilon)\frac{(n!)^{1/n}}{n}\,\V(K) 
$$
where $\varepsilon>0$ is small enough to satify all estimates below. In particular,
$\varepsilon<\frac1{4n^3}\,\tilde{\varepsilon}_0$, where 
$\tilde{\varepsilon}_0$ comes from  Theorem~\ref{Henk-Linke-stab1}.
Applying  \eqref{sigmainduction} for $m=1$,  \eqref{sigmainduction0} for $m\geq 2$,
and using $(1+\varepsilon)^n\frac{n-1}{n}<\frac{n-1}{n}+2n\varepsilon$
gives
\begin{equation}
\label{sigmainduction1}
\int_{S^{n-1}}(\V(K)-\V_K(S^{n-1}\cap{\rm lin}\{u\}))\,d\V_K(u)\leq   \left(\frac{n-1}{n}+2n\varepsilon\right)\V(K)^2.
\end{equation}
For any $X\subset S^{n-1}$, there exists $u\in X$ maximizing $\V_K(S^{n-1}\cap{\rm lin}\{u\})$ because different $1$-dimensional subspaces have disjoint intersections with $S^{n-1}$.
We consider linearly independent $v_1,\ldots,v_n\in S^{n-1}$  such that $v_1$ maximizes $\V_K(S^{n-1}\cap{\rm lin}\{u\})$
for $u\in S^{n-1}$, and  $v_i$ maximizes $\V_K(S^{n-1}\cap{\rm lin}\{u\})$
for all $u\in S^{n-1}\backslash{\rm lin}\{v_1,\ldots,v_{i-1}\}$ if $i=2,\ldots,n$. Let $L={\rm lin}\{v_1,\ldots,v_{n-1}\}$,
and let $\V_K(S^{n-1}\cap{\rm lin}\{v_n\})=(\frac1n-t)\V(K)$, and hence
$t\in[0,\frac1n]$ (cf.~\eqref{scc}). Thus we have 
\begin{eqnarray}
\label{VKvi}
\V_K(S^{n-1}\cap{\rm lin}\{v_i\})&\geq &\mbox{$(\frac1n-t)\V(K)$  \ for $i=1,\ldots,n$},\\
\label{VKvnu}
\V_K(S^{n-1}\cap{\rm lin}\{u\})&\leq &\mbox{$(\frac1n-t)\V(K)$  \ for $u\in S^{n-1}\backslash L$.}
\end{eqnarray}
We deduce from (\ref{sigmainduction1}),  (\ref{VKvnu}) and
$\V_K(S^{n-1}\cap{\rm lin}\{u\})\leq \frac1n\,\V(K)$ for $u\in S^{n-1}\cap L$ that
$$
(\mbox{$\frac{n-1}n$}+t)\V(K)\V_K(S^{n-1}\backslash L)+
\mbox{$\frac{n-1}n$}\,\V(K)\V_K(S^{n-1}\cap L)\leq   \left(\mbox{$\frac{n-1}{n}$}+2n\varepsilon\right)\V(K)^2.
$$
Since $\V_K(S^{n-1}\backslash L)\geq\frac1n\,\V(K)$ according to Theorem~\ref{Henk-Linke-body}, we conclude that
$t\leq 2n^2\varepsilon$. In particular, $\V_K(S^{n-1}\cap{\rm
  lin}\{v_i\})\geq (\frac1n-2n^2\varepsilon)\V(K)$ 
for $i=1,\ldots,n$ by (\ref{VKvi}).

From  Theorem~\ref{Henk-Linke-stab1} we find for $i=1,\ldots,n$,
that there exist an 
 $(n-1)$-dimensional compact convex set $C_i\subset v_i^\bot$ with $c(C_i)=o$, and $x_i,y_i\in \partial K$ 
such that $y_i=-e^{s_i}x$, where $|s_i|<n\tilde{\gamma}_v\varepsilon^{\frac1{6}}$, and for $i=1,\ldots,n$, we have
\begin{eqnarray}
\label{CixyinK}
[x_i,y_i]+C_i&\subset &K \\
\label{CixyKvol}
\V(K\backslash([x_i,y_i]+C_i))&\leq &n\tilde{\gamma}_v\varepsilon^{\frac1{6}}\V(K)\\
\label{CixyKhom}
K&\subset& [x_i,y_i]+(1+2\tilde{\gamma_h}\varepsilon^{\frac1{6n}})C_i.
\end{eqnarray}
We may assume that $v_i$ is an exterior normal at $x_i$, $i=1,\ldots,n$. After a linear transformation of $K$, we may also assume that $v_1,\ldots,v_n$ form and orthonormal system, and $\langle v_i,x_i-y_i\rangle=2$. In particular,
\begin{equation}
\label{xynorm}
e^{-\tau}<\langle v_i,x_i\rangle,\langle -v_i,y_i\rangle< e^{\tau},\;\tau=n\tilde{\gamma}_v\varepsilon^{\frac1{6}}.
\end{equation}
In what follows, we write $\gamma_1,\gamma_2,\ldots$ for positive constants depending on $n$ only.
It follows from combining (\ref{CixyinK}), (\ref{CixyKvol}) and (\ref{xynorm}) that
\begin{equation}
\label{CiCjarea}
1-\gamma_1\varepsilon^{\frac1{6}}<{\mathcal H}_{n-1}(C_i)/{\mathcal H}_{n-1}(C_j)<
1+\gamma_1\varepsilon^{\frac1{6}}\mbox{ \ for $i,j\in\{1,\ldots,n\}$}.
\end{equation}
For any $i\neq j\in\{1,\ldots,n\}$, we write
\begin{eqnarray*}
w_i(v_j)&=&h_{C_i}(v_j)+h_{C_i}(-v_j),\\
a_i(v_j)&=&\max\left\{ {\mathcal H}_{n-2}(C_i\cap (tv_j+v_j^\bot)):\,-h_{C_i}(-v_j)\leq t\leq h_{C_i}(v_j)\right\},
\end{eqnarray*}
and recall that $h_{C_i}(x)$ denotes the support function.
Hence $w_i(v_j)$ is the width of $C_i$ in the direction of $v_j$. Calculating 
${\mathcal H}_{n-1}(C_i)$ by integrating along $\R v_j$ leads to 
\begin{equation}
\label{Ciwjajarea}
\mbox{$\frac1{n-1}$}\,w_i(v_j)a_i(v_j)\leq {\mathcal H}_{n-1}(C_i)\leq w_i(v_j)a_i(v_j)
\mbox{ \ for $i\neq j\in\{1,\ldots,n\}$}.
\end{equation}

Let $p\neq q\in\{1,\ldots,n\}$. We choose $t_1\geq t_*\geq t_0$ such that
\begin{eqnarray*}
\langle v_p,t_1x_p\rangle &=& h_{C_q}(v_p)\\
\langle -v_p,t_0x_p\rangle &=& h_{C_q}(-v_p)\\
{\mathcal H}_{n-2}(C_q\cap (t_*x_p+v_p^\bot))&=&a_q(v_p).
\end{eqnarray*}
It follows from (\ref{xynorm}) and (\ref{CixyKhom})  that
\begin{eqnarray}
\label{t1t0diff}
t_1-t_0& > &w_q(v_p)/2,\\
\nonumber
C_q\cap (t_*x_p+v_p^\bot)&\subset & t_*x_p+(1+2\tilde{\gamma_h}\varepsilon^{\frac1{6n}})C_p.
\end{eqnarray}
Therefore $a_p(v_q)\geq (1+2\tilde{\gamma_h}\varepsilon^{\frac1{6n}})^{-(n-2)}a_q(v_p)$,
and hence interchanging the role of $p$ and $q$ leads to
$$
1-\gamma_2\varepsilon^{\frac1{6n}}<a_q(v_p)/a_q(v_p)<
1+\gamma_2\varepsilon^{\frac1{6n}}.
$$
We deduce from (\ref{CiCjarea}) and (\ref{Ciwjajarea}) that
\begin{equation}
\label{wpwq}
\frac1{2n}<\frac{w_p(v_q)}{w_q(v_p)}<2n.
\end{equation}
Now combining (\ref{CixyinK}) and (\ref{CixyKhom}) shows that
\begin{equation}
\label{hpvq}
h_{C_p}(v_q)\leq \langle x_q-t_mx_p,v_q\rangle \leq (1+2\tilde{\gamma_h}\varepsilon^{\frac1{6n}})h_{C_p}(v_q)
\mbox{ \ for $m=0,1$},
\end{equation}
and hence
$$
|\langle(t_1-t_0)x_p,v_q\rangle|\leq 2\tilde{\gamma_h}\varepsilon^{\frac1{6n}} h_{C_p}(v_q)
<2\tilde{\gamma_h}\varepsilon^{\frac1{6n}} w_p(v_q).
$$
Applying \eqref{t1t0diff}, \eqref{wpwq}, and the analoguous argument to $y_q$ implies that
\begin{equation}
\label{xpvq}
|\langle x_p,v_q\rangle|,|\langle y_p,v_q\rangle|\leq \gamma_3\varepsilon^{\frac1{6n}}.
\end{equation}

Let $P$ be the parallepiped
$$
P=\{x\in\R^n:\,\langle x,v_i\rangle\leq \langle x_i,v_i\rangle,\;\langle x,-v_i\rangle\leq \langle y_i,-v_i\rangle,\;
i=1,\ldots,n\},
$$
and hence each facet of $P$ contains one of $x_i+C_i$, $y_i+C_i$, $i=1,\ldots,n$.
We claim that
\begin{equation}
\label{Pclaim}
\mbox{$\frac1{4n}$}\,P\subset K.
\end{equation}
We suppose that (\ref{Pclaim}) does not hold, and seek a contradiction. Possibly reversing the orientation of some of the $v_i$, we may asssume that
\begin{equation}
\label{Pclaimcontra}
z=\frac1{4n}\sum_{i=1}^n\langle x_i,v_i\rangle\, v_i\not\in K.
\end{equation}
In particular, $\|z\|\leq \frac1{2\sqrt{n}}$ by (\ref{xynorm}), and there exists $u\in S^{n-1}$ such that
\begin{equation}
\label{Pclaimu}
\langle u,z\rangle>\langle u,x\rangle\mbox{ \ for $x\in K$.}
\end{equation}
There exists $v_p$ such that $|\langle u,v_p\rangle|\geq 1/\sqrt{n}$, and hence  (\ref{xynorm}) and (\ref{xpvq}) yield that
$\langle u,x_p\rangle\geq\frac1{\sqrt{n}}-\gamma_4\varepsilon^{\frac1{6n}}$ if $\langle u,v_p\rangle\geq 1/\sqrt{n}$,
and $\langle u,y_p\rangle\geq\frac1{\sqrt{n}}-\gamma_4\varepsilon^{\frac1{6n}}$ if $\langle u,v_p\rangle\leq -1/\sqrt{n}$.
However $\langle u,z\rangle\leq \|z\|\leq\frac1{2\sqrt{n}}$, contradicting (\ref{Pclaimcontra}). Therefore we conclude 
(\ref{Pclaim}).

For $i=1,\ldots,n$, let
$$
\Xi_{2i-1}=[o,x_i+C_i]\mbox{ \ and \ } \Xi_{2i}=[o,y_i+C_i].
$$
Since the basis of the cones $\Xi_1,\ldots,\Xi_{2n}$ lie in different
facets of $P$, the interiors of $\Xi_1,\ldots,\Xi_{2n}$ are pairwise
disjoint. By (\ref{CixyKvol}) and (\ref{xynorm}) we know
$\V(\Xi_j)\geq (\frac1{2n}-\gamma_5\varepsilon^{\frac1{6}})\V(K)$, and
so we get 
$$
\V(\Xi)>(1-2n\gamma_5\varepsilon^{\frac1{6}})\V(K)\mbox{ \ for $\Xi=\bigcup_{j=1}^{2n}\Xi_j\subset K$}.
$$
We conclude from (\ref{Pclaim}) that
$$
\V(P\backslash K)\leq \V(P\backslash \Xi)=(4n)^n V\left((\mbox{$\frac1{4n}$}\,P)\backslash \Xi\right)
\leq (4n)^n V\left(K\backslash \Xi\right)\leq \gamma_6\varepsilon^{\frac1{6}}\V(K).
$$
Therefore (\ref{CixyKvol}) yields
$$
{\mathcal H}_{n-1}((P\cap v_i^\bot)\backslash C_i)\leq \gamma_7\varepsilon^{\frac1{6}}{\mathcal H}_{n-1}(C_i),\;
i=1,\ldots,n,
$$
and Lemma~\ref{centroidhom} implies that $(1-\gamma_8\varepsilon^{\frac1{6n}})P\subset K$,
completing the proof of Theorem~\ref{U(K)stab}.


\end{document}